\documentclass[reqno]{amsart}
\usepackage{graphicx} 
\usepackage{subfiles}
\usepackage{enumitem}
\usepackage[T1]{fontenc}
\usepackage{dsfont}
\newtheorem{theorem}{Theorem}[section]
\newtheorem{lemma}{Lemma}[section]
\newtheorem{proposition}{Proposition}[section]

\newtheorem{definition}{Definition}[section]
\theoremstyle{remark}
\newtheorem{remark}{Remark}[section]
\usepackage{physics}
\usepackage{mathtools}
\usepackage{amsmath, amssymb}
\calclayout
\def\undertilde#1{\mathord{\vtop{\ialign{##\crcr
$\hfil\displaystyle{#1}\hfil$\crcr\noalign{\kern1.5pt\nointerlineskip}
$\hfil\widetilde{}\hfil$\crcr\noalign{\kern1.5pt}}}}}
\setlist[enumerate,2]{label=\arabic*)} \newcommand*\diff{\mathop{}\!\mathrm{d}}
\newcommand{\R}{\mathbb{R}}
\newcommand{\eps}{\epsilon}
\newcommand{\N}{\mathbb{N}}
\newcommand{\ind}{\mathds{1}}
\newcommand{\aplim}{\textnormal{ap }\lim}
\DeclarePairedDelimiter{\ceil}{\lceil}{\rceil}
\newcommand{\statespace}{\mathcal{V}_0}
\newcommand{\weak}{\mathcal{S}_{\text{weak}}}
\allowdisplaybreaks

\title[weak-BV stability for isothermal gas dynamics]{Uniqueness \& Weak-BV stability in the large for isothermal gas dynamics}

\author{Jeffrey Cheng}
\address{Department of Mathematics, The University of Texas at Austin, Austin, TX 78712.}
\email{jeffrey.cheng@utexas.edu}
\date{\today}
\thanks{\textbf{Acknowledgements}: The author would like to thank Sam Krupa \& Alexis Vasseur for the suggestion of the problem, and Cooper Faile for the short proof of Lemma \ref{productlemma}. This work was partially supported by NSF grants: DMS-1840314 \& DMS-2306852 and a joint NSF-ESPRC grant: DMS-EPSRC-2219434}

\begin{document}
\emergencystretch 3em

\begin{abstract}
For the $1$-d isothermal Euler system, we consider the family of entropic BV solutions with possibly large, but finite, total variation. We show that these solutions are stable with respect to large perturbations in a class of weak solutions to the system which may not even be BV. The method is based on the construction of a modified front tracking algorithm, in which the theory of $a$-contraction with shifts for shocks is used as a building block. The main contribution is to construct the weight in the modified front tracking algorithm in a large-BV setting.
\end{abstract}
\maketitle 

\tableofcontents
\section{Introduction \& main results}
In this paper, we consider the $1-$d isothermal Euler system (i.e. isentropic Euler with $\gamma=1$)
\begin{align}\label{eulerian}
\begin{cases}
\partial_t \rho+\partial_y(\rho v)=0, & (t,y) \in \R^+ \times \R,\\
\partial_t (\rho v)+\partial_y(\rho v^2+P(\rho))=0, & (t,y) \in \R^+ \times \R, \\
\end{cases}
\end{align}
where $\rho$ is the density, $v$ is the velocity, and $P(\rho)=\rho$. If one makes a change of spatial variables to the Lagrangian mass coordinate $x=\int_{y(t)}^y\rho(s,t)\diff s$, where $y(t)$ is a well-defined particle path with $y'(t)=v(y(t),t)$, one obtains the $p$-system
\begin{align}\label{isothermal}
\begin{cases}
\partial_t \tau-\partial_xv=0, & (t,x) \in \R^+ \times \R,\\
\partial_t v+\partial_xp(\tau)=0, & (t,x) \in \R^+ \times \R,\\
\end{cases}
\end{align}
where $\tau=\frac{1}{\rho}$ is the specific volume and $p(\tau)=\frac{1}{\tau}$. The two formulations are equivalent for bounded weak solutions bounded away from vacuum \cite{MR885816}. Both systems \eqref{eulerian}\eqref{isothermal} are special cases of $1-$d hyperbolic systems of conservation laws
\begin{align}\label{system}
\partial_tu+\partial_xf(u)=0,
\end{align}
endowed with a strictly convex entropy $\eta$ and entropy flux $q$ which obey the functional equation $\nabla q=\nabla \eta \nabla f$. 
Specifically, for \eqref{eulerian} we have
\begin{align}
\begin{cases}\label{eulerianentropystructure}
\tilde{\eta}(\rho,v)=\frac{\rho v^2}{2}+\rho \ln(\rho)-\rho, \\
\tilde{q}(\tau, v)=\frac{\rho v^3}{2}+\rho v \ln(\rho),
\end{cases}
\end{align}
while for \eqref{isothermal} we have
\begin{align}
\begin{cases}\label{entropystructure}
\eta(\tau,v)=\frac{v^2}{2}-\ln(\tau), \\
q(\tau, v)=p(\tau)v=\frac{v}{\tau}.
\end{cases}
\end{align}
In this article, we will work primarily on the $p$-system and our main result will be stated for \eqref{isothermal}. We assume $u=(\tau,v) \in \statespace \subset \R^2$, where $\statespace$ is as follows
\begin{align*}
    \statespace=\{u=(\tau,v) \in \R^+ \times \R: |v|\leq \hat{M}, \hat{c} \leq \tau \leq \hat{C} \},
\end{align*}
for some $\hat{M},\hat{C},\hat{c} > 0$. 
The system \eqref{isothermal} has eigenvalues
\begin{align}\label{evs}
\lambda_{1,2}(u)=\mp \frac{1}{\tau}.
\end{align}
We will consider only entropic solutions of \eqref{isothermal}, i.e. solutions $u$ which satisfy additionally
\begin{align}\label{entropyinequality}
\partial_t\eta(u)+\partial_x q(u) \leq 0, \indent (t,x) \in \R^+ \times \R,
\end{align}
with respect to the entropy \eqref{entropystructure}. Without ambiguity, we may also refer to solutions to $\eqref{eulerian}$ verifying \eqref{entropyinequality}, by which we mean solutions that verify \eqref{entropyinequality} with respect to the entropy \eqref{eulerianentropystructure}. We also restrict ourselves to solutions satisfying the Strong Trace property.
\begin{definition}\label{strongtrace}
Let $u \in L^\infty(\R^+ \times \R)$. We say that $u$ verifies the Strong Trace property if for any Lipschitzian curve $t \mapsto h(t)$, there exists two bounded functions $u_-, u_+ \in L^\infty(\R^+)$ such that for any $T > 0$:
\begin{align*}
&\lim_{n\to \infty}\int_0^T\sup_{y \in (0,\frac{1}{n})}|u(t,h(t)+y)-u_+(t)|\diff t \\
&=\lim_{n\to \infty}\int_0^T\sup_{y \in (-\frac{1}{n},0)}|u(t,h(t)+y)-u_-(t)|\diff t=0.
\end{align*}
\end{definition}
With this in mind, we may define the largest space of solutions we consider in this paper.
\begin{align*}
\weak:=\{u \in L^\infty(\R^+ \times \R;\statespace)|\text{ weak solution to \eqref{isothermal}\eqref{entropyinequality}, satisfying Definition \ref{strongtrace}}\}.
\end{align*}
Note that any BV function satisfies Definition \ref{strongtrace}. Thus, any BV solution to \eqref{isothermal}\eqref{entropyinequality} valued in $\statespace$ belongs to the class $\weak$. The aim of the paper is to show the stability of BV solutions in the larger class $\weak$. Our main result is the following theorem.
\begin{theorem}\label{main}
Let $u=(\tau,v) \in \weak \cap L^\infty(\R^+;BV(\R))$ be an entropic BV solution to \eqref{isothermal}\eqref{entropyinequality} with initial data $u^0=(\tau^0,v^0)$. Let $u_n=(\tau_n,v_n) \in \weak$ be a sequence of wild solutions with initial values $u_n^0=(\tau_n^0, v_n^0)$. If $u_n^0$ converges to $u^0$ in $L^2(\R)$, then for any $T,R>0$, $u_n$ converges to $u$ in $L^\infty([0,T];L^2([-R,R]))$. In particular, $u$ is the unique function in the class $\weak$ with initial data $u^0$.
\end{theorem}
As a corollary, due to the equivalence of the Eulerian and Lagrangian formulations \cite{MR885816} we obtain the following theorem in the Eulerian frame. Define the state space and solution spaces analogously:
\begin{align}
&\statespace^E:=\{(\rho, \rho v): |v| \leq \tilde{M}, \tilde{c} \leq \rho \leq \tilde{C} \text{ for some $\tilde{M}, \tilde{c}, \tilde{C} > 0$}\}, \\
&\weak^E:=\{u \in L^\infty(\R^+ \times \R;\statespace^E)|\text{ weak solution to \eqref{eulerian}\eqref{entropyinequality}, satisfying Definition \ref{strongtrace}}\}.
\end{align}
Then, we have
\begin{theorem}\label{maineulerian}
Let $U=(\rho, \rho v) \in \weak^E \cap L^\infty(\R^+; BV(\R))$ be an entropic BV solution to \eqref{eulerian}\eqref{entropyinequality} with initial data $U^0=(\rho^0, \rho^0v^0)$. Let $U_n=(\rho_n,\rho_nv_n) \in \weak^E$ be a sequence of wild solutions with initial values $U^0_n=(\rho^0_n, \rho^0_nv^0_n)$. If $U^0_n$ converges to $U^0$ in $L^2(\R)$, then for any $T,R>0$, $U_n$ converges to $U$ in $L^\infty([0,T];L^2([-R,R]))$. In particular, $U$ is unique in the class $\weak^E$. 
\end{theorem}
Note that there is no hypothesis in Theorem \ref{main} that the perturbations $u_n$ be BV. Only the solution $u$ need be well behaved. The current state-of-the-art uniqueness theorem for the isothermal Euler system gives uniqueness in the class of entropic BV solutions (see Theorem \ref{uniqueness}). Theorem \ref{main} extends this to the wider class of entropic solutions verifying Definition \ref{strongtrace}.

\par The study of the $1-$d Euler equations dates back to Riemann, who studied the eponymous ``Riemann problem'' \cite{MR1630338}. Since then, the existence and uniqueness theory for $1$-d hyperbolic systems has been greatly developed. Glimm proved small-BV existence for general hyperbolic systems via a random choice algorithm \cite{MR194770}. The front tracking scheme was developed by Dafermos \cite{MR303068} in the scalar setting and has been extended to small-BV existence for $n \times n$ systems \cite{MR1816648}, \cite{MR3468916}. Small-BV existence was also shown via the vanishing viscosity method in the seminal paper of Bianchini \& Bressan \cite{MR2150387}. Initially, uniqueness was proven in the class of BV functions under various regularity conditions \cite{MR1701818}, \cite{MR1489317}, \cite{MR1757395} (for a comprehensive account, see \cite{MR1816648}). For systems endowed with a strictly convex entropy, all of these conditions were recently removed (for $2 \times 2$ systems in \cite{MR4487515}, and for $n \times n$ systems in \cite{MR4690554}, \cite{MR4661213}).
\par In regards to stability, the $L^1$-stability of small-BV weak solutions was established in the 90s \cite{MR1686652}, \cite{MR1723032}. These theories require that the perturbations still have small-BV. In contrast, an $L^2$-stability theory can handle large perturbations which may not even be BV. The first results in this line of work were by Dafermos \cite{MR546634} \& DiPerna \cite{MR523630}, who showed that for systems endowed with a strictly convex entropy, Lipschitz solutions are ``very'' stable in the following sense. For any $T > 0$, consider the classes of weak solutions:
\begin{align*}
&\mathcal{S}^T_{\text{reg}}=\{u \in L^\infty(\R \times [0,T];\mathcal{V}_0), \text{ solution to \eqref{system}, with $||\partial_xu(t)||_{L^\infty(\R)} \leq C$ for some $C > 0$}\}, \\
&\mathcal{S}^T_{\text{weak}}=\{u \in L^\infty(\R \times [0,T];\mathcal{V}_0), \text{ solution to \eqref{system}\eqref{entropyinequality}}\}.
\end{align*}
Let $u$ be a solution in $\mathcal{S}^T_{\text{reg}}$ with initial data $u^0$. The results of Dafermos \& DiPerna imply that if $\{u_n\}_{n \in \N}$ is a sequence of solutions in $\mathcal{S}^T_{\text{weak}}$ such that their initial values $\{u_n^0\}_{n \in \N}$ converge to $u^0$ in $L^2(\R)$, then $\{u_n\}_{n \in \N}$ converges to $u$ in $L^\infty([0,T];L^2(\R))$. This grants the ``weak-strong'' stability principle; strong (Lipschitz) solutions are stable in the much wider class of weak solutions $\mathcal{S}^T_{\text{weak}}$. 
\par A recent line of work has extended the program of Dafermos \& DiPerna to discontinuous solutions. The first result was to show $L^2$-stability for shocks of scalar conservation laws \cite{MR2771666}. This was extended to the system setting (for extremal shocks) by Kang \& Vasseur \cite{MR3519973}. Using this result as a building block, the ``weak-strong'' stability was recently extended to a ``weak-small-BV'' stability principle for general $2 \times 2$ systems by Chen, Krupa, \& Vasseur \cite{MR4487515}. It says that small-BV weak solutions are stable in the larger class of weak solutions $\weak$. The strategy of the proof is to construct a modified front tracking algorithm alongside an appropriate weight function.
\par The present paper is concerned with extending the ``weak-small-BV'' principle to a ``weak-BV'' principle for the isothermal Euler system. The isothermal Euler system is special in the sense that discretization schemes, such as that of Glimm or the front tracking algorithm, are well-posed for large-BV initial data. This is essentially because the system is endowed with a ``total variation decreasing'' (TVD) field, a scalar field whose variation is non-increasing along solutions. Systems with TVD fields were studied by Bakhvalov \cite{MR279443}, but are somewhat limited in scope (see \cite{MR3040678} for a discussion of this. In particular, the isentropic Euler system does not have a TVD field for $\gamma > 1$). The large-BV existence was first shown by Nishida for the Glimm scheme \cite{MR236526}, and by Asakura for the front tracking algorithm \cite{MR2126567} (see also \cite{MR1359913} for an extension to the isothermal Euler-Poisson system). We show that such constructions can be modified so that the strategy carried out in \cite{MR4487515} can be extended to the large-BV setting. This is partially motivated by the recent execution of this extension in the scalar setting \cite{concaveconvex}. 
\par This work is also motivated by the construction of $L^\infty$ solutions for $2 \times 2$ systems, including the isothermal Euler system, via the compensated compactness theory \cite{MR584398} (see \cite{MR719807}, \cite{MR924671}, \cite{MR922139}, \cite{MR1284790}, \cite{MR1383202} for $\gamma \neq 1$, and \cite{MR1970885} for $\gamma=1$). Theorem \ref{main} says that BV solutions bounded away from vacuum are $L^2$-stable in the class $\weak$, which includes the solutions constructed via compensated compactness (as long as they verify Definition \ref{strongtrace}).
\par The paper is structured as follows. In Section \ref{prelim}, we give some uniqueness and compactness theorems with origins in the $L^1$-theory for later use. Then, we give some specific details on the structure of the TVD field and shock \& rarefaction curves for \eqref{isothermal}. In Section \ref{strat}, we give the main proposition, Proposition \ref{mainprop}, and show how Theorem \ref{main} follows from Proposition \ref{mainprop}. The rest of the paper is aimed towards proving Proposition \ref{mainprop}. In Section \ref{riemannentropy}, we summarize $L^2$-stability results for single waves from \cite{MR4667839}, \cite{MR4487515}. The modified front tracking algorithm is constructed in Section \ref{mfts}, and the weight is constructed in Section \ref{weight}. Finally, in Section \ref{final}, we prove Proposition \ref{mainprop}.

\section{Preliminaries}\label{prelim}
For use later, we record here a uniqueness theorem and a compactness theorem originating in the $L^1$-theory. The first result is the synthesis of two results in the literature, rephrased in our context, which establish the existence and uniqueness of the ``Standard Riemann Semigroup'' for the isothermal Euler system.
\begin{theorem}[Synthesis of main theorems from \cite{MR1641709}, \cite{MR4690554}]\label{uniqueness}
Fix $M \geq 0$. Then, there exists only one solution $u:\R^+ \times \R \to \statespace$ to \eqref{isothermal}\eqref{entropyinequality} (or equivalently, \eqref{eulerian}\eqref{entropyinequality}) such that $||u(t)||_{BV(\R)} \leq M$ for all $t \geq 0$.
\end{theorem}
\begin{remark}
The theorems of \cite{MR1641709}, \cite{MR4690554} are stated in an $L^1$ framework, but trajectories of limits of front tracking approximations may be defined for functions simply with finite total variation (and the solution operator remains Lipschitz w.r.t. the $L^1$ distance) and so the uniqueness theory works with $L^\infty$ solutions without issue.
\end{remark}
The next theorem is a compactness theorem which we will need to extract a convergent subsequence from the sequence of piecewise constant functions generated by our modified front tracking scheme. It was adapted from the classical scheme in \cite{MR1816648} to the modified scheme in \cite{MR4487515}.
\begin{theorem}[Lemma 2.6 in \cite{MR4487515}]\label{compactness}
Let $\{\psi_n\}_{n \in \N}$ be a family of piecewise constant functions uniformly bounded in $L^\infty(\R^+;BV(\R))$. Assume that there exists $C>0$ such that for every $n \in \N$
\begin{align}\label{lipintime}
||\psi_n(t,\cdot)-\psi_n(s,\cdot)||_{L^1(\R)} \leq C|t-s|, \indent 0<s<t<T.
\end{align}
Then, there exists $\psi \in L^\infty(\R^+ \times \R)$ such that, up to a subsequence, $\psi_n$ converges to $\psi$ as $n \to \infty$ in $C([0,T];L^2([-R,R]))$ for every $T,R>0$, and almost everywhere in $\R^+ \times \R$. 
\end{theorem}

\subsection{The shock \& rarefaction curves and the TVD field}
Here, we give a recount of some useful structural results for $\eqref{isothermal}$, following \cite{MR1942468}. Firstly, we give the structure of the $i$-wave curves. Let $w=-\ln(\tau)$. Define a function:
\begin{align*}
W(s)=\begin{cases}
s & s \leq 0, \\
2\sinh(\frac{s}{2}) & s \geq 0. \\
\end{cases}
\end{align*}
Then, the equations for the $i$-wave curves passing through a point $(w_l, v_l)$ corresponding to a left state $u_l$, where $w_l=-\ln(\tau_l)$, are:
\begin{align}\label{wavecurves}
\text{$i$-wave curve:} \indent \sigma \mapsto T_i(\sigma)((w_l,v_l)):=(w_l+(-1)^{i-1}\sigma, -W(\sigma)+v_l), \indent i=1,2.
\end{align}
A right state $u_r=(\tau_r,v_r)$ such that $(w_r,v_r)=T_i(\sigma)(w_l,v_l)$ for may be connected to $u_l$ via an $i$-rarefaction if $\sigma < 0$, while a right state may be connected to $u_l$ via an $i$-shock if $\sigma > 0$. Note that the shock and rarefaction curves are translation invariant in the $(w,v)$-plane. The large data existence theory follows from the fact that $w$ is a TVD field for $\eqref{isothermal}$.  For this reason, we will find it convenient to do our analysis of wave interactions in the $(w,v)$-plane rather than the $(\tau,v)$-plane, and measure the variation of a piecewise constant function via its variation in $w$. Henceforth, we may refer to a state either in $(\tau,v)$ or $(w,v)$ coordinates interchangeably for convenience.
\par Let $(w_m, v_m)$ be the intermediate state in the Riemann problem connecting the states $(w_l, v_l)$ and $(w_r, v_r)$. Note that $(w_m,v_m)$ is uniquely defined via the following relation
\begin{align}\label{intermediatestate}
    v_l-v_r=W(w_m-w_l)+W(w_m-w_r).
\end{align}
Suppressing the dependence of the states on $v$ for convenience, we define:
\begin{align*}
    D(w_l,w_r):=|w_l-w_m|+|w_m-w_r|.
\end{align*}
Then, we have the following lemma.
\begin{lemma}[Proposition 3.1 in \cite{MR1359913}]\label{tvd}
For any states $w_i$, $i=1,2,3$, we have
\begin{align*}
    D(w_1,w_3) \leq D(w_1,w_2)+D(w_2,w_3).
\end{align*}
\end{lemma}
Taking $w_2$ to be the intermediate state between $w_1$ and $w_3$ before a wave interaction, this shows that $w$ is TVD over pairwise wave interactions.

\section{Weighted relative entropy and shifts}\label{strat}
Our proof is based on the relative entropy method. Fix a compact set $\mathcal{D} \subset \R^+ \times \R$. Using the entropy \eqref{entropystructure}, we define an associated pseudo-distance, called the relative entropy, for any $a,b \in \mathcal{D} \times \mathcal{D}$:
\begin{align}\label{relativeentropy}
\eta(a|b)=\eta(a)-\eta(b)-\nabla\eta(b) \cdot (a-b).
\end{align}
We also define the relative entropy-flux:
\begin{align}\label{relativeflux}
q(a;b)=q(a)-q(b)-\nabla \eta(b) \cdot (f(a)-f(b)).
\end{align}
Now, for any $b \in \mathcal{D}$ constant, if $u$ verifies \eqref{isothermal}\eqref{entropyinequality}, then $u$ further verifies
\begin{align}\label{relativeentropic}
\partial_t \eta(u|b)+\partial_x q(u;b) \leq 0.
\end{align}
Similar to Kru\v{z}kov's theory for scalar conservation laws, this provides a full family of entropies for the system \eqref{isothermal}. However, these entropies measure the square of the $L^2$ distance rather than the $L^1$ distance between two states. Indeed, we have the following result.
\begin{lemma}[from \cite{MR2807139}, \cite{MR2508169}]\label{quadraticentropy}
For any fixed compact set $\mathcal{D}$, there exists $c^*, c^{**} > 0$ such that for all $a,b \in \mathcal{D} \times \mathcal{D}$:
\begin{equation*}
    c^*|a-b|^2 \leq \eta(a|b)\leq c^{**}|a-b|^2
\end{equation*}
The constants $c^*, c^{**}$ depend on bounds on the second derivative of $\eta$ in $\mathcal{D}$, and on the continuity of $\eta$ on $\mathcal{D}$.
\end{lemma}
Studying the time evolution of the relative entropy for Lipschitz solutions yields the weak-strong principle. However, the study of shocks requires a more intricate theory, the so-called ``a-contraction with shifts''. This theory was used to show the weak-BV stability in \cite{MR4487515}. Their main proposition (Proposition 3.2) shows that if a solution $u \in \weak$ is close to a small-BV function in $L^2$ at $t=0$, it will still be close to a small-BV function in $L^2$ later in time. The main goal of this article is to remove the smallness from their proposition for the system \eqref{isothermal}.
\par In what follows, we will work with two compact subsets of $\R^+ \times \R$. The first is $\statespace$, the compact set in which solutions to $\eqref{isothermal}$ are valued. The second will be a compact set $\mathcal{D}$ that depends on the initial data $u^0$ in the following proposition. This is needed because, although $u^0$ may be initially valued in $\statespace$, the failure of the maximum principle means that the functions $\psi$ we construct need not be valued in $\statespace$ for $t > 0$. Without loss of generality, we may assume $\statespace \subset \mathcal{D}$.
\begin{proposition}\label{mainprop}
Consider the system $\eqref{isothermal}$. Then, for any $u^0 \in BV(\R) \cap L^\infty(\R)$ with $\tau_0 \geq \beta >0$, there exists a compact set $\mathcal{D} \subset \R^+ \times \R$ and $C',l > 0$ such that the following is true. For any $m > 0$, $R,T > 0$, and $u \in \weak$, there exists $\psi:\R^+ \times \R \to \mathcal{D}$ such that for almost every $0 < s < t < T$: 

\begin{align*}
    &||\psi(t,\cdot)||_{BV(\R)} \leq C'||u^0||_{BV(\R)}, \\
    &||\psi(t,\cdot)-\psi(s,\cdot))||_{L^1(\R)} \leq C'|t-s|, \\
    &||\psi(t,\cdot)-u(t,\cdot)||_{L^2([-R+lt,R-lt])} \leq C'\left(||u^0-u(0,\cdot)||_{L^2([-R,R])}+\frac{1}{m}\right).
\end{align*}
The domain $\mathcal{D}$ and constants $C'$ and $l$ depend only on $||u^0||_{BV(\R)}, ||u^0||_{L^\infty(\R)}$, and $\beta$.
\end{proposition}
The domain $\mathcal{D}$ and the functions $\psi$ satisfying Proposition \ref{mainprop} will be constructed in Section \ref{mfts} via a modified front tracking scheme. Due to the shifts, they will not be solutions to \eqref{isothermal}. In fact, the only thing we will know about them is an estimate on their BV norm. So Proposition \ref{mainprop} simply propagates an estimate on the distance to a BV function in time. Nevertheless, it is enough to grant Theorem \ref{main}.
\par The main difficulty is in defining the constant $C'$ to satisfy the third property. The $L^2$-stability for shocks is given by a contraction in a weighted relative entropy distance, where the weight $a$ depends on the strength of the shock. So the weight must be updated every time two waves interact, and can grow or decay greatly as many shocks interact or are produced. The main property we need to show is that the weight cannot degenerate to zero or blow up to infinity through this process. We finish this section by showing how Proposition \ref{mainprop} implies Theorem \ref{main}.
\newline \textit{Proof Of Theorem \ref{main}.} \indent Fix $T,R > 0$. Passing to a subsequence if needed, assume that $||u^0_m-u^0||_{L^2(\R)} \leq \frac{1}{m}$. By Proposition \ref{mainprop}, there exists a sequence of functions $\{\psi_m\}_{m \in \N}$ uniformly bounded in $L^\infty(\R^+;BV(\R))$ such that \eqref{lipintime} holds uniformly in $m$. Moreover, for all $0< t < T$, and any $R' > 0$: 
\begin{align}\label{close}
    ||\psi_m(t,\cdot)-u_m(t,\cdot)||_{L^2([-R'+lt,R'-lt])} \leq \frac{2C'}{m}.
\end{align}
By Theorem \ref{compactness}, there exists $\psi \in L^\infty(\R^+, BV(\R))$ such that $\psi_m$ converges in $C([0,T];L^2([-R,R]))$ to $\psi$. Taking $R'$ sufficiently large in $\eqref{close}$, we have that $u_m$ converges in $L^\infty([0,T];L^2([-R,R]))$ to $\psi$. Since the convergence is strong, and each $u_m$ verifies \eqref{isothermal}\eqref{entropyinequality}, so does $\psi$. Further, $\psi$ is valued in $\statespace$ as all the $u_m$ are. So, by Theorem \ref{uniqueness}, it is the unique solution to \eqref{isothermal} verifying the uniform BV bound. So, $\psi=u$. This completes the proof of Theorem \ref{main}. The remainder of the paper is devoted to proving Proposition \ref{mainprop}.

\begin{remark}
One may prove Theorem \ref{maineulerian} from Proposition \ref{mainprop} by doing the change of coordinates at the level of the front-tracking approximations in the proof of Proposition \ref{mainprop} (cf. Proposition 6.1 in \cite{MR3322780}). For convenience, a proof is given in Appendix \ref{appendixeulerian}.
\end{remark}

\section{Relative entropy for the Riemann problem}\label{riemannentropy}
In this section, we give the relative entropy stability results for shocks and rarefactions. The first proposition, the relative entropy stability for shocks, is a summary of the results of \cite{MR4184662} and \cite{MR4667839}. We measure the strength of a shock by its jump in $w$.
For convenience, given a Lipschitz function $h: \R^+ \to \R$, we will use the notation $u_+(t)=u(t,h(t)+)$, and $u_-(t)=u(t,h(t)-)$. 
\begin{proposition}[Cf. Proposition 5.1 in \cite{MR4184662} and Theorem 1 in \cite{MR4667839}]\label{shockstability}
Fix a compact set $\mathcal{D} \subset \R^+ \times \R$. There exist constants $0 < C_0 < \frac{1}{2},0< C_1 < 1, \eps>0,$ and $\hat{\lambda}>0$ such that the following is true.
\par Let $u \in \weak$, $\overline{t} \in [0,\infty)$, and $x_0 \in \R$.  Consider any entropic shock $(u_l,u_r)$ such that $u_l, u_r \in \mathcal{D}$. Let $\sigma$ be the strength of the shock $\sigma=|w_l-w_r|$. Then, if $\sigma \geq \frac{\eps}{2}$, for any $a_1, a_2 > 0$ satisfying:
\begin{align*}
&\frac{a_2}{a_1} \leq C_1, \indent \text{ if $(u_l,u_r)$ is a $1$-shock}, \\
&\frac{a_2}{a_1} \geq \frac{1}{C_1}, \indent \text{ if $(u_l,u_r)$ is a $2$-shock},
\end{align*}
there exists a Lipschitz shift $h_1:[\overline{t},\infty) \to \R$ with $h(\overline{t})=x_0$ such that the following dissipation functional verifies:
\begin{equation*}
a_1(\dot{h}_1(t)\eta(u_-(t)|u_l)-q(u_-(t);u_l))-a_2(\dot{h}_1(t)\eta(u_+(t)|u_r)-q(u_+(t);u_r)) \leq 0,
\end{equation*}
for almost every $t \in [\overline{t},\infty)$. On the other hand, if $\sigma < \eps$, for any $a_1, a_2 > 0$ verifying:
\begin{align*}
& 1-2C_0\sigma \leq \frac{a_2}{a_1} \leq 1-\frac{C_0}{2}\sigma, \text{ if $(u_l,u_r)$ is a $1$-shock}, \\
& 1+\frac{C_0\sigma}{2} \leq \frac{a_2}{a_1} \leq 1+2C_0\sigma, \text{ if $(u_l,u_r)$ is a $2$-shock},
\end{align*}
there exists a Lipschitz shift $h_2:[\overline{t},\infty) \to \R$ with $h_2(\overline{t})=x_0$ such that the following dissipation functional verifies:
\begin{equation*}
a_1(\dot{h}_2(t)\eta(u_-(t)|u_l)-q(u_-(t);u_l))-a_2(\dot{h}_2(t)\eta(u_+(t)|u_r)-q(u_+(t);u_r)) \leq 0,
\end{equation*}
for almost every $t \in [\overline{t},\infty)$. Moreover, if $(u_l,u_r)$ is a $1-$shock, then for almost all $t \in [\overline{t},\infty)$:
\begin{align*}
    -\hat{\lambda} \leq \dot{h}_{1,2}(t) < 0.
\end{align*}
Similarly, if $(u_l,u_r)$ is a $2-$shock, then for almost all $t \in [\overline{t},\infty)$:
\begin{align*}
    0 < \dot{h}_{1,2}(t) \leq \hat{\lambda}.
\end{align*}
\end{proposition}
Notice that there is an interval where the two regimes overlap:  for shocks with strength $\frac{\eps}{2} \leq \sigma < \eps$, we may have the non-positivity of the dissipation functional with either choice of weight. This will be important in Section \ref{weight}. 
\begin{remark}
Note that in the articles \cite{MR4184662} and \cite{MR4667839}, the shock strength is measured by Euclidean distance while in this context, we measure the shock strength by the jump in $w$. This is equivalent due to the specific structure of the shock curves and the uniform continuity of the natural log on the compact set $\mathcal{D}$.
\end{remark}
\begin{remark}
Proposition 5.1 in \cite{MR4184662} and Theorem 1 in \cite{MR4667839} prove the results of Proposition \ref{shockstability} locally around a fixed shock, while Proposition \ref{shockstability} holds globally for shocks with both states in the compact set $\mathcal{D}$. The compactness argument to move from the local statements to the global one is simple. The interested reader is referred to \cite{concaveconvex}, where this is done in the scalar setting.
\end{remark}
\begin{remark}
By examining the proofs of Proposition 5.1 in \cite{MR4184662} and Theorem 1 in \cite{MR4667839}, one notes that $C_1$ and $\eps$ may be taken arbitrarily small if desired, and $C_0$ may be taken arbitrarily large if desired. In particular, we may assume that $C_0\eps \leq \frac{1}{2}$.\end{remark}
For the rarefactions, the following proposition follows from the fact that rarefaction waves for $\eqref{isothermal}$ verify a relative entropy contraction without the need for a shift. A proof may be found in \cite{MR4487515}.
\begin{proposition}[Proposition 4.4 in \cite{MR4487515}]\label{rarefactionstability}
There exists a constant $\Theta>0$ such that the following is true. For any rarefaction wave $\overline{u}(y), s_l \leq y \leq s_r$, with $\overline{u}(s_l)=u_l, \overline{u}(s_r)=u_r$, denote:
\[
\kappa=|s_l-s_r|+\sup_{y \in [s_l,s_r]}|u_l-\overline{u}(y)|.
\]
Then, for any $u \in \weak$, any $s_l \leq s \leq s_r$, and any $t > 0$, we have:
\begin{align*}
&\int_0^t\left(q(u(t,st+);u_r)-q(u(t,st-);u_L)-s(\eta(u(t,st+)|u_r)-\eta(u(t,st-)|u_l))\right)\diff t \\
&\leq \Theta \kappa |u_l-u_r|t.
\end{align*}
\end{proposition}

\section{Modified front tracking scheme}\label{mfts}
In this section, we introduce the modified front tracking scheme. Briefly, the idea is as follows. One approximates a given initial data $u^0$ by a piecewise constant function. Then, one approximately solves the Riemann problems that arise and lets the waves propagate. We then solve Riemann problems on demand whenever two waves interact. A more detailed exposition is given in Section \ref{construction}. We refer to the foundational book of Bressan \cite{MR1816648} for a general introduction to the front tracking algorithm for $n \times n$ systems. However, as $\eqref{isothermal}$ is a $2 \times 2$ system, it is well-known that one need not use the ``non-physical waves'' developed by Bressan \cite{MR3001706}. Further, in this paper, we do not use the front tracking algorithm to generate a solution to the conservation law. We only use the fact that the algorithm gives a sequence of functions with uniformly bounded total variation. 

\par In Section \ref{solver0} we describe the Riemann solver for constructing approximate solutions to the Riemann problem. In our solver, the solution will depend on a classification of incoming and outgoing shocks as either ``big'' or ``small''. In Section \ref{solver0}, we also describe how to classify shocks at $t=0$. In Section \ref{solvert}, we give an explicit list of all possible two-wave interactions at $t=t^* > 0$ alongside the respective classifications of shocks and useful estimates on various quantities when the interactions occur. Although the first description is a much more succinct way to describe the solution to the Riemann problem, the case-by-case analysis in Section \ref{solvert} will be needed for fine estimates on the weight $a$ in Section \ref{weight}. 
\subsection{The Riemann solver}\label{solver0}
In this section, we describe how to solve Riemann problems for \eqref{isothermal}. Fix a small parameter $0 < \delta < \frac{\eps}{2}$. Given two constant states $u_l=(\tau_l,v_l)$ and $u_r=(\tau_r, v_r)$, recall from $\eqref{wavecurves}$ that there are two smooth curves in the $(w,v)$-plane
\begin{align*}
\sigma \mapsto T_i(\sigma)((w_l,v_l)):=(w_l+(-1)^{i-1}\sigma, -W(\sigma)+v_l), \indent i=1,2,
\end{align*}
parameterized by $w$ such that
\begin{align}\label{Riemannsolution}
    (w_r, v_r)=T_2(\sigma_2) \circ T_1(\sigma_1)((w_l,v_l)),
\end{align}
for some $\sigma_1$ and $\sigma_2$. We emphasize that the strength of the waves is measured by the jump in $w$ and that we solve Riemann problems in the $(w,v)$-plane. Note that if $\sigma_i$ is positive (negative) the states are separated by an $i$-shock ($i$-rarefaction). 
\par Now, assume that at a time $\overline{t}$, there is an interaction at the point $\overline{x}$. Let $u_l, u_r$ be the Riemann problem generated by the interaction and let $\sigma_1, \sigma_2$ be the solution defined by \eqref{Riemannsolution}. Define
\begin{align*}
    &u_0:=u_l, \\
    &u_1:=(\tau_1, v_1) \text{ such that } (w_1, v_1)=T_1(\sigma_1)((w_l,v_l)), \\
    &u_2:=u_r. 
\end{align*}
If $\sigma_i < 0$, then we let:
\begin{align*}
    p_i:=\ceil{\frac{\sigma_i}{\delta}},
\end{align*}
where $\ceil{s}$ denotes the smallest integer greater than or equal to $s$. For $l=1, ..., p_i$, we define
\begin{align}
    u_{i,l}:=(\tau_{i,l}, v_{i,l}) \text{ such that } (w_{i,l}, v_{i,l})=T_i(\frac{l\sigma_i}{p_i})((w_{i-1}, v_{i-1})), \indent x_{i,l}(t):=\overline{x}+(t-\overline{t})\lambda_i(u_{i,l}),
\end{align}
while if $\sigma_i > 0$, we define $p_i=1$ and
\begin{align}
    u_{i,l}=u_i, \indent x_{i,l}(t)=h_{k}(t),
\end{align}
where $k=1$ if the shock is ``big'', and $k=2$ if the shock is ``small'' (at $t=0$, all shocks with $\sigma_i \geq \eps$ will be classified as big, and all shocks with $\sigma_i < \eps$ will be classified as small. The classification of shocks arising from an interaction at $t=t^*>0$ will be discussed in Section \ref{solvert}). Here, $h_k$ are the shifts coming from Proposition \ref{shockstability}. Then, the approximate solution to the Riemann problem is:
\begin{align}\label{approxriemannsolution}
    v_\alpha(t,x)=
\begin{cases}
  u_l & x < x_{1,1}(t), \\
  u_r &  x > x_{2,p_2}(t),  \\
  u_i & x_{i,p_i}(t)<x<x_{i+1,1}(t), \\
  u_{i,l} & x_{i,l}<x<x_{i,l+1}(t), \indent (l=1, ..., p_{i-1}). 
\end{cases}
\end{align}
\begin{remark}
The function is well-defined as $x_{i,p_i}(t) < x_{i+1,1}(t)$ due to the separation of wave speeds in Proposition \ref{shockstability}.
\end{remark}

\subsection{Classification of shocks and estimates for the Riemann solver at $t=t^*>0$}\label{solvert}
In this section, we describe all possible two-wave interactions which occur at a time $t=t^* > 0$ and the classification of shocks in the corresponding solution of the Riemann problem. There are $6$ possible waves
\begin{center}
\begin{tabular}{ c c c }
a) big $1$-shock & b) small $1$-shock & c) big $2$-shock \\ 
d) small $2$-shock & e) $1$-rarefaction & f) $2$-rarefaction \\    
\end{tabular}
\end{center}
These give rise to $21$ possible distinct interactions (discounting which wave is on the left or on the right). They can be sorted into categories as follows.
\begin{center}
\begin{tabular}{ |c|c|c| } 
 \hline
Impossible & Trivial & Non-trivial \\ 
e)+e) & a)+c) & a)+a) \\ 
f)+f) & a)+d) & a)+b) \\ 
      & a)+f) & a)+e) \\
      & b)+c) & b)+b) \\
      & b)+d) & b)+e) \\
      & b)+f) & c)+c) \\
      & c)+e) & c)+d) \\
      & d)+e) & c)+f) \\
      & e)+f) & d)+d) \\
      &       & d)+f) \\
 \hline
\end{tabular}
\end{center}
The interactions labeled ``impossible'' cannot occur due to the wave speeds, and the ``trivial'' interactions are labeled as such because the corresponding Riemann solution is simply the waves passing through each other (although the intermediate state may change). So we need only consider the non-trivial interactions. Recall that we measure the strength of waves by their jump in $w$. So, given a front tracking pattern defined in $[0,T] \times \R$, define the following quantities:
\begin{align}
    &S(t)=\sum_{i: \text{small shock}}\sigma_i, \\
    &B(t)=\sum_{i: \text{wave}}|\sigma_i|.
\end{align}
$S$ measures the total strength of all small shocks present in the solution at time $t$, while $B$ measures the total variation in $w$. Firstly, note that these quantities may only change at a time when waves interact. Further note that by Lemma \ref{tvd}, $B$ can never increase when two waves interact; it can only decrease or remain the same. 
\par Now, consider a non-trivial interaction at $t=t^* > 0$ between a wave $(u_l, u_m)$ and a wave $(u_m,u_r)$. Henceforth, we suppress the dependence of states on $v$ and reference them only by $w$ for convenience. Denote by $w_m'$ the intermediate state in the Riemann solution between $w_l$ and $w_r$, determined via \eqref{intermediatestate}. We go though each non-trivial interaction and record the changes in $S(t)$ and $B(t)$ when the interaction occurs. Without loss of generality, as the wave curves are invariant under translation, for an interaction of type m)+n), assume that $(w_l, w_m)$ is the wave of type m), and $(w_m, w_r)$ is the wave of type n). \newline
\newline 
\underline{\textbf{(1): a)+a)}} The solution will be a big $1$-shock composed with a possibly partitioned $2$-rarefaction. In this case, $\Delta S(t^*):=S(t^*+)-S(t^*-)=0$, and $\Delta B(t^*)$=0.
\newline 
\underline{\textbf{(2): a)+b)}} The solution will be a big $1$-shock composed with a possibly partitioned $2$-rarefaction. In this case, $\Delta S(t^*)=-(w_r-w_m) < 0$, and $\Delta B(t^*)$=0. 
\newline
\underline{\textbf{(3): a)+e)}} The solution will be a $1$-shock composed with a small $2$-shock. Here, we split into two cases. 
\newline
\indent \underline{\textbf{(3A)}} If $(w_l,w_m')$ has strength $\sigma \geq \frac{\eps}{2}$, it will retain the label ``big'', and the solution will be a big $1$-shock composed with a small $2$-shock. Here, $\Delta S(t^*)=w_m'-w_r > 0$, and $\Delta B(t^*)=-2(w_m-w_m') < 0$.
\newline 
\indent \underline{\textbf{(3B)}} If $(w_l,w_m')$ has strength $\sigma < \frac{\eps}{2}$, it will be labeled ``small'', and the solution will be a small $1$-shock composed with a small $2$-shock. Here, $\Delta S(t^*)=w_m'-w_l+w_m'-w_r \leq \eps$, and $\Delta B(t^*)=-2(w_m-w_m') < 0$.
\newline
\underline{\textbf{(4): b)+b)}} The solution will be a $1$-shock composed with a possibly partitioned $2$-rarefaction. Here, we split into two cases.
\newline
\indent \underline{\textbf{(4A)}} If $(w_l,w_m')$ has strength $\sigma < \eps$, it will remain ``small'', and the solution will be a small $1$-shock composed with a possibly partitioned $2$-rarefaction. Here, $\Delta S(t^*)=-(w_r-w_m') < 0$, and $\Delta B(t^*)=0$.
\newline
\indent \underline{\textbf{(4B)}} If $(w_l,w_m')$ has strength $\sigma \geq \eps$, it will become ``big'', and the solution will be a big $1$-shock composed with a possibly partitioned $2$-rarefaction. Here, $\Delta S(t^*)=-(w_r-w_l) \leq -\eps$, and $\Delta B(t^*)=0$.
\newline
\underline{\textbf{(5): b)+e)}} Here, we split into two cases.
\newline
\indent \underline{\textbf{(5A)}} If $(w_l, w_m')$ is a $1-$shock, the solution will be a small $1$-shock composed with a small $2$-shock. Here, $\Delta S(t^*)=2w_m'-w_m-w_r$, while $\Delta B(t^*)=-2(w_m-w_m')<0$.
\newline
\indent \underline{\textbf{(5B)}} If $(w_l, w_m')$ is a $1-$rarefaction, the solution will be a $1-$rarefaction composed with a small $2$-shock. Here, $\Delta S(t^*)=w_l-w_m+w_m'-w_r \leq 0$, while $\Delta B(t^*)=-2(w_m-w_l)<0$.
\newline
\par The next set of interactions concern interactions involving $2$-shocks. The interaction $i$ will be essentially just the reflection of the interaction $i-5$ (e.g. type 1 interactions are comparable to type 6 ones).
\newline 
\newline
\underline{\textbf{(6): c)+c)}} The solution will be a possibly partitioned $1$-rarefaction composed with a big $2$-shock. In this case, $\Delta S(t^*)=0$, and $\Delta B(t^*)$=0.
\newline 
\underline{\textbf{(7): c)+d)}} The solution will be a possibly partitioned $1$-rarefaction fan composed with a big $2$-shock. In this case, $\Delta S(t^*)=w_r-w_m < 0$, and $\Delta B(t^*)$=0.
\newline
\underline{\textbf{(8): c)+f)}} The solution will be a small $1$-shock composed with a $2$-shock. Here, we split into two cases. 
\newline
\indent \underline{\textbf{(8A)}} If $(w_m',w_r)$ has strength $\sigma \geq \frac{\eps}{2}$, it will retain the label ``big'', and the solution will be a small $1$-shock composed with a big $2$-shock. Here, $\Delta S(t^*)=w_m'-w_l > 0$, and $\Delta B(t^*)=-2(w_l-w_m-w_m'+w_r) < 0$.
\newline 
\indent \underline{\textbf{(8B)}} If $(w_m',w_r)$ has strength $\sigma < \frac{\eps}{2}$, it will be labeled ``small'', and the solution will be a small $1$-shock composed with a small $2$-shock. Here, $\Delta S(t^*)=w_m'-w_l+w_m'-w_r \leq \eps$, and $\Delta B(t^*)=-2(w_l-w_m-w_m'+w_r) < 0$.
\newline
\underline{\textbf{(9): d)+d)}} The solution will be a possibly partitioned $1$-rarefaction composed with a $2$-shock. Here, we split into two cases.
\newline
\indent \underline{\textbf{(9A)}} If $(w_m',w_r)$ has strength $\sigma < \eps$, it will remain ``small'', and the solution will be a possibly partitioned $1$-rarefaction composed with a small $2$-shock. Here, $\Delta S(t^*)=w_l-w_m' < 0$, and $\Delta B(t^*)=0$.
\newline
\indent \underline{\textbf{(9B)}} If $(w_m', w_r)$ has strength $\sigma \geq \eps$, it will become ``big'', and the solution will be a possibly partitioned small $1$-rarefaction composed with a big $2$-shock. Here, $\Delta S(t^*)=-(w_l-w_r) \leq -\eps$, and $\Delta B(t^*)=0$.
\newline
\underline{\textbf{(10): d)+f)}} Here, we split into two cases.
\newline
\indent \underline{\textbf{(10A)}} If $(w_m', w_r)$ is a $2-$shock, the solution will be a small $1$-shock composed with a small $2$-shock. Here, $\Delta S(t^*)=2w_m'-2w_l-w_r+w_m$, while $\Delta B(t^*)=-2(w_l-w_m+w_r-w_m') < 0$.
\newline
\indent \underline{\textbf{(10B)}} If $(w_m', w_r)$ is a $2$-rarefaction, the solution will be a small $1-$shock composed with a $2$-rarefaction. Here, $\Delta S(t^*)=w_m'-w_l-(w_l-w_m) \leq 0$, while $\Delta B(t^*)=-2(w_l-w_m) < 0$.

\subsection{Construction of the wave pattern}\label{construction}
The scheme proceeds as follows. Recall that in the context of Proposition \ref{mainprop}, one starts with initial data $u^0 \in BV(\R) \cap L^\infty(\R)$ and $\tau_0 \geq \beta$. Fix a small parameter $\delta > 0$. Firstly, one approximates the initial data $u^0$ by a piecewise constant initial data $\psi^0_\delta=(\tau^0_{\delta}, v^0_{\delta})$ such that:
\begin{align}\label{desired}
    &||\psi^0_\delta||_{BV(\R)} \leq 2||u^0||_{BV(\R)}, \notag \\
    &||\psi^0_\delta||_{L^\infty(\R)} \leq ||u^0||_{L^\infty(\R),}, \\
    &\tau^0_{\delta}(x) \geq \beta, \indent x \in \R,\notag 
\end{align}
For example, one may fix $R>0$ and take averages on a partition of $[-R,R]$ with mesh size $\delta$ as the approximating functions. Then, the second and third properties of \eqref{desired} will hold, and the first will hold for $\delta$ sufficiently small. This will additionally grant:
\begin{align*}
||\psi^0_\delta-u^0||_{L^2([-R,R])} \to 0 \text{ as } \delta \to 0.
\end{align*}
Then at $t=0$, we solve the Riemann problems according to the solution dictated in Section \ref{solver0} and let the waves propagate. All shocks with strength $\sigma \geq \eps$ are initially labeled as ``big'', while all shocks with strength $\sigma < \eps$ are initially labeled as ``small''. After this, when any waves interact at a time $t=t^* > 0$, we locally solve the Riemann problem and classify outgoing shocks according to the classification dictated in Section \ref{solvert}. We repeat this ad infinitum, and define a function $\psi_\delta(t,x)=(\tau_\delta(t,x),v_\delta(t,x))$ as the piecewise constant function generated by this process. We also define $w_\delta(t,x):=-\ln(\tau_\delta(t,x))$. Note that by slightly perturbing the speed of waves, we may assume that at each time there is at most one collision involving only two waves.
\par We need to show that for any finite time $T>0$, the front tracking algorithm is well-defined in $[0,T] \times \R$, i.e. that the number of wave interactions does not explode, and also that the total variation and $L^\infty$ norm are uniformly bounded in $\delta$. This is summarized in the following lemma. 
\begin{lemma}\label{wellposed}
Let $T > 0$. Then, the number of interactions in the region $[0,T] \times \R$ is finite. Further, there exist a constant $M=M(||u^0||_{L^\infty(\R)}, ||u^0||_{BV(\R)}, \beta)>0$ and a compact set $\mathcal{D}=\mathcal{D}(||u^0||_{L^\infty(\R)}, ||u^0||_{BV(\R)}, \beta) \subset \R^+ \times \R$
such that for any $\delta > 0$
\begin{align*}
    &\psi_\delta(t,x) \in \mathcal{D} \indent \text{ for any } (t,x) \in \R^+ \times \R, \\
    &||\psi_\delta||_{L^\infty([0,T];BV(\R))} \leq M||u^0||_{BV(\R)}, \\
    &||\psi_\delta(t,\cdot)-\psi_\delta(s,\cdot)||_{L^1(\R)} \leq M|t-s|, \indent 0<s<t<T. \\
\end{align*}
\end{lemma}
A proof is given in Appendix \ref{appendix}. The next two lemmas will be used to control the decay of the weight function defined in Section \ref{weight}. Let $V$ be an upper bound on $||w_\delta(0+,\cdot)||_{BV(\R)}$, uniform in $\delta$, which only depends on $||u^0||_{L^\infty(\R)}, ||u^0||_{BV(\R)}$, and $\beta$ (cf. Appendix \ref{appendix}).
\begin{lemma}\label{bigsmall}
For any front tracking pattern in $[0,T] \times \R$, with initial data $\psi_\delta^0$ satisfying \eqref{desired}, when a big shock turns into a small shock, $B$ has decreased by at least $\eps$. Thus, this cannot happen more than $\ceil{\frac{V}{\eps}}$ times. In particular, the number of type 3B and 8B interactions is bounded by $\ceil{\frac{V}{\eps}}$.
\end{lemma}
\begin{proof}
Recall that shocks will only become big once they have strength $\sigma \geq \eps$. Then, the strength of a big shock may only decrease through an interaction of type 3 or 8. Notice that in those interactions, the decrease of $B$ is exactly twice the decrease in the strength of the big shock. So, in order for a big shock to become small, it must lose at least $\frac{\eps}{2}$ strength, which implies $B$ has decreased by at least $\eps$. Thus, this cannot happen more than $\ceil{\frac{V}{\eps}}$ times. 
\end{proof}

\begin{lemma}\label{Sdissipationbd}
Define
\begin{align*}
\Delta_+S(t):=\max(0,\Delta S(t)).
\end{align*}
Then, for any front tracking pattern in $[0,T] \times \R$, with initial data $\psi_\delta^0$ satisfying \eqref{desired}, we have a bound
\begin{align*}
\sum_{t>0} \Delta_+S(t) \leq K_0,
\end{align*}
for some $K_0>0$, depending only on $||u^0||_{L^\infty(\R)}, ||u^0||_{BV(\R)}$, and $\beta$. Further, we also have
\begin{align*}
\sum_{t>0} \Delta_-S(t) \leq K_0+V,
\end{align*}
where $\Delta_-S(t)=-\min(0,\Delta S(t))$. In particular, the number of type 4B and 9B interactions is bounded by $\ceil{\frac{K_0+V}{\eps}}$.
\end{lemma}
\begin{proof}
Assume that an interaction in which $\Delta S(t^*) > 0$ occurs at $t=t^*$. This must be a type 3A, 3B, 5A, 8A, 8B, or 10A interaction. Without loss of generality (as interactions of type $i$ and $i-5$ are simply reflections of one another), we will assume it is of type 3A, 3B or 5A. In all three cases, we will show that $\Delta S(t^*)$ is controlled by a decrease in $B$.
\par Firstly, assume it is a type 3A interaction. Then, $\Delta S(t^*)=w_m'-w_r$, while $\Delta B(t^*)=-2(w_m-w_m')$. Denote by $k_1$ the point on the $2$-shock curve emanating from $w_m'$ with $v$-coordinate $v_m$, and denote by $k_2$ the point on the $1$-shock curve emanating from $w_m'$ with $v$-coordinate $v_m$. Let $w_{1,2}$ be the projections of $k_{1,2}$ onto the $w-$axis. Then, we have:
\begin{align}\label{step1}
    &w_m'-w_r < w_m'-w_1, \notag \\
    &w_2-w_m'=w_m'-w_1.
\end{align}
Further, as shown in the proof of Lemma \ref{wellposed}, there exists $K_1 > 0$ such that $|w_\delta(t,x)| \leq K_1$ for any $t,x \in [0,T] \times \R$. Using $1 \leq |W'(s)| \leq C_3$ for $|s| \leq 2K_1$, we obtain
\begin{align}\label{step2}
    w_2-w_m' \leq C_3(w_m-w_m').
\end{align}
Together, \eqref{step1} and \eqref{step2} imply
\begin{align}\label{step3}
    w_m'-w_r < C_3(w_m-w_m').
\end{align}
But \eqref{step3} implies
\begin{align*}
\Delta S(t^*) \leq C_3|\Delta B(t^*)|.
\end{align*}
So, we obtain
\begin{align}\label{case3A}
\sum_{t:\text{type 3A}}\Delta S(t) \leq C_3V
\end{align}
Now, assume the interaction is of type 3B. Then, $\Delta S(t) \leq \eps$, but by Lemma \ref{bigsmall}, this may only occur at most $\ceil{\frac{V}{\eps}}$ times. So, by taking $\eps \leq V$
\begin{align}\label{case3B}
\sum_{t:\text{type 3B}}\Delta S(t) \leq \eps \ceil{\frac{V}{\eps}} \leq V+\eps \leq 2V.
\end{align}
The case when the interaction is of type 5A is similar to the 3A case. Here, we obtain
\begin{align}\label{case5A}
\sum_{t:\text{type 5A}}\Delta S(t) \leq (C_3+1)V  .
\end{align}
Consolidating \eqref{case3A}, \eqref{case3B}, and \eqref{case5A}, we obtain
\begin{align*}
\sum_{t > 0}\Delta S(t) \leq 2(C_3+1)V.
\end{align*}
This grants the result with $K_0=2(C_3+1)V$. Finally, if a big shock is created, in either a type 4B or 9B interaction, $\Delta S(t^*)\leq -\eps$. Thus, we conclude that this can only happen at most $\ceil{\frac{K_0+V}{\eps}}$ times.
\end{proof}

\section{Construction of the weight $a$}\label{weight}
In this section, we construct the weight function that will be used to measure the stability. The weight will be defined as follows.
\begin{align*}
a(t,x)=\left(C_1^{L(t)}\right)\left(Q(t)\right)\left(\displaystyle \prod_{i:\text{shock}}\xi_i(t,x)\right),
\end{align*}
where the functions $L(t), Q(t), $ and $\xi_i(t,x)$ are defined as follows.
\newline \underline{\textbf{L(t):}} The function $L(t)$ is defined as 
\begin{align*}
    L(t)=\sum \limits_{\mathclap{\substack{0 < t^* <t: \text{type 1, 3B,} \\
    \text{or 9B interaction at $t^*$}}}}1.
\end{align*}
The value of the function $L$ simply goes up by $1$ every time an interaction of type 1, 3B, or 9B occurs.
\newline \underline{\textbf{Q(t):}} The function $Q(t)$ is defined as 
\begin{align*}
    Q(t)=\displaystyle \prod_{k_i \in K(t)}k_i,
\end{align*}
where the set $K(t)$ is defined as follows. Initially, $K(0+)=\emptyset$. The set $K(t)$ consists of ``small shock numbers'' accumulated as waves interact at $t=t^* >0$ in the following manner. Recall that without loss of generality we assume that in an interaction of type m)+n), the wave $w_l, w_m$ is of type m), and the wave $w_m,w_r$ is of type n).
\begin{enumerate}
\item If the interaction is of type 2, $K(t^*+)=K(t^*-) \cup \{1-C_0(w_r-w_m)\}$, where $C_0$ is the constant in Proposition \ref{shockstability}.
\item If the interaction is of type 3A,  $K(t^*+)=K(t^*-) \cup \{\frac{1}{1+C_0(w_m'-w_r)}\}$.
\item If the interaction is of type 4A,  $K(t^*+)=K(t^*-) \cup \{\min{(\frac{(1-C_0(w_m-w_l))(1-C_0(w_r-w_m))}{1-C_0(w_m'-w_l)},1)}\}$.
\item If the interaction is of type 5A,  $K(t^*+)=K(t^*-) \cup \{\frac{1-C_0(w_m-w_l)}{(1-C_0(w_m'-w_l))(1+C_0(w_m'-w_r))}\}$.
\item If the interaction is of type 5B, $K(t^*+)=K(t^*-) \cup \{\frac{1-C_0(w_m-w_l)}{1+C_0(w_m'-w_r)}\}$.
\end{enumerate}
\underline{$\boldsymbol{\xi_i(t,x)):}$} Finally, a shock wave $x_i(t)$ with strength $\sigma_i$ has the associated function $\xi_i(t,x)$:
\begin{align*}
\xi_i(t,x)=\begin{cases}
            \ind_{\{x < x_i(t)\}}(x)+(1-C_0\sigma_i)\ind_{\{x > x_i(t)\}}(x) & \text{if $x_i(t)$ is a small 1-shock}, \\
            \ind_{\{x < x_i(t)\}}(x)+(1+C_0\sigma_i)\ind_{\{x > x_i(t)\}}(x) & \text{if $x_i(t)$ is a small 2-shock}, \\
            \ind_{\{x < x_i(t)\}}(x)+C_1\ind_{\{x > x_i(t)\}}(x) & \text{if $x_i(t)$ is a big 1-shock}, \\
            \ind_{\{x < x_i(t)\}}(x)+\frac{1}{C_1}\ind_{\{x > x_i(t)\}}(x) & \text{if $x_i(t)$ is a big 2-shock}.
         \end{cases}
\end{align*}
Clearly, the component of the function $a$ involving the weights $\xi_i$ is the portion that grants the $L^2$-stability. The other two components are needed in order for the weight to be decreasing in time across wave interactions (this is proved later, in Proposition \ref{aproperties}).

The next two lemmas give bounds on the first two components of the function $a$ uniform in $\delta$. 
\begin{lemma}\label{Lproperties}
The function $L$ is monotonically increasing in time. Further, if there is a wave interaction at a time $t=t^* > 0$
\begin{align}\label{deltaL}
    \Delta L(t^*)=
    \begin{cases}
        1 &  \text{ interaction is of type 1, 3B, or 9B} \\
        0 &  \text{ otherwise.}
    \end{cases}
\end{align}
Finally, there is a bound $L(t) \leq \Lambda$ for some $\Lambda > 0$ uniform in $\delta$.
\end{lemma}
\begin{proof}
The property \eqref{deltaL} is obvious from the definition. Next, we show that $L$ is uniformly bounded in time. We show that the number of possible interactions of types 1, 3B, and 9B respectively is bounded uniformly in $\delta$. 
\par Whenever a type 1 interaction occurs, the number of big shocks goes down by 1. However, the number of times a big shock can be created is finite. Indeed, a big shock can only be created in a type 4B or 9B interaction. By Lemma \ref{Sdissipationbd}, we conclude that this can only happen at most $\ceil{\frac{K_0+V}{\eps}}$ times. Thus, the number of type 1 interactions is bounded by $\ceil{\frac{V}{\eps}}+\ceil{\frac{K_0+V}{\eps}}$.
\par Next, the number of type 3B interactions is bounded by $\ceil{\frac{V}{\eps}}$ by Lemma \ref{bigsmall}.
\par Finally, by Lemma \ref{Sdissipationbd}, a type 9B interaction may occur maximally $\ceil{\frac{K_0+V}{\eps}}$ times. So, there is a bound on $L(t)$ uniform in $\delta$
\begin{align*}
    L(t) \leq 2(\ceil{\frac{K_0+V}{\eps}}+\ceil{\frac{V}{\eps}}).
\end{align*}
\end{proof}

\begin{lemma}\label{Qproperties}
The function $Q$ is monotonically decreasing in time. Further, there is a bound $Q(t) \geq k$ for some $k > 0$ uniform in $\delta$. 
\end{lemma}
Before we begin the proof, we give a lemma that will uniformly bound contributions in $Q(t)$ coming from small shocks.
\begin{lemma}\label{productlemma}
Let $\{a_i\}_{i=1}^N$ be a finite set of real numbers with $0 < a_i < \frac{1}{2}$ for all $1 \leq i \leq N$, and $\sum_{i=1}^Na_i \leq K$ for some $K > 0$. Then:
\[
\prod_{i=1}^N(1-a_i) \geq \exp({-2K}),
\]
and
\[
\prod_{i=1}^N(1+a_i) \leq \exp(K).
\]
\end{lemma}
\begin{proof}
Using $\exp(-2x) \leq 1-x$ for $0 \leq x \leq \frac{1}{2}$, we obtain
\begin{align*}
    \prod_{i=1}^N(1-a_i) \geq \prod_{i=1}^N\exp(-2a_i)=\exp(-2\sum_{i=1}^Na_i)\geq \exp(-2K).
\end{align*}
On the other hand, using $\exp(x) \geq 1+x$ gives
\begin{align*}
    \prod_{i=1}^N(1+a_i) \leq \exp(K).
\end{align*}
\end{proof}
Now, we give the proof of Lemma \ref{Qproperties}.
\newline \underline{\textbf{Proof of Lemma \ref{Qproperties}:}} The function $Q$ is obviously monotonically decreasing in time as at any interaction time at which a weight is added to $K$, the weight added is less than or equal to one. To show the bound from below, we consider each possible interaction type separately and show that contributions to $Q$ coming from that interaction type are bounded from below.
\begin{enumerate}
    \item Firstly, assume a weight is added via a type 2 interaction. Then, the weight added is $1-C_0(w_r-w_m)$, while $\Delta S(t)=-(w_r-w_m)$. By Lemmas \ref{Sdissipationbd} and \ref{productlemma}, we conclude that for any $t> 0$
    \begin{align*}
        \prod\limits_{\mathclap{\substack{k_i \text{ comes from} \\
        \text{type 2 interaction}}}}k_i \geq \exp(-2C_0(K_0+V)).
    \end{align*}
    \item Assume a weight is added via a type 3A interaction. Then, the weight added is $\frac{1}{1+C_0(w_m'-w_r)}$, while $\Delta B(t)=-2(w_m-w_m')$. Writing $\frac{1}{1+C_0(w_m'-w_r)}=1-y$, we see
    \begin{align*}
    y=\frac{C_0(w_m'-w_r)}{1+C_0(w_m'-w_r)} \leq C_0(w_m'-w_r).
    \end{align*}
    By \eqref{step3} and Lemma \ref{productlemma}, we conclude that for any $t> 0$
    \begin{align*}
        \prod\limits_{\mathclap{\substack{k_i \text{ comes from} \\
        \text{type 3A interaction}}}}k_i \geq \exp(-2C_0C_3V).
    \end{align*}
    \item Assume a weight is added via a type 4A interaction. Then, if the weight added is less than $1$, it is $\frac{(1-C_0(w_m-w_l))(1-C_0(w_r-w_m))}{1-C_0(w_m'-w_l)}$, while $\Delta S(t)=-(w_r-w_m')$. Writing the weight as $1-y$, we see
    \begin{align*}
        y \leq 2C_0(w_r-w_m').
    \end{align*}
    By Lemma \ref{productlemma}, we conclude
    \begin{align*}
        \prod\limits_{\mathclap{\substack{k_i \text{ comes from} \\
        \text{type 4A interaction}}}}k_i \geq \exp(-4C_0(K_0+V)).
    \end{align*}
    \item Assume a weight is added via a type 5A interaction. Then, the weight added is $\frac{1-C_0(w_m-w_l)}{(1-C_0(w_m'-w_l))(1+C_0(w_m'-w_r))}$, while $\Delta B(t)=-2(w_m-w_m')$. Writing the weight as $1-y$, we see
    \begin{align*}
        y \leq 4C_0((w_m'-w_r)+(w_m-w_m')).
    \end{align*}
    Identically to the proof of \eqref{step3}, one may show
    \begin{align*}
        w_m'-w_r \leq C_3 (w_m-w_m'),
    \end{align*}
    so we conclude by Lemma \ref{productlemma}
    \begin{align*}
        \prod\limits_{\mathclap{\substack{k_i \text{ comes from} \\
        \text{type 5A interaction}}}}k_i \geq \exp(-8C_0C_3V).        
    \end{align*}
    \item Finally, assume a weight is added via a type 5B interaction. Then, the weight added is $\frac{(1-C_0(w_m-w_l))}{1+C_0(w_m'-w_r)}$, while $\Delta B(t)=-2(w_m-w_l)$. Writing the weight as $1-y$, 
    \begin{align*}
        y \leq 4C_0(w_m-w_l),
    \end{align*}
    where we have used $w_m'-w_r \leq w_m-w_l$. So we conclude by Lemma \ref{productlemma}
    \begin{align*}
        \prod\limits_{\mathclap{\substack{k_i \text{ comes from} \\
        \text{type 5B interaction}}}}k_i \geq \exp(-4C_0V).        
    \end{align*}
\end{enumerate}
Finally, combining all the cases
\begin{align*}
    Q(t)=\prod_{k_i \in K(t)}k_i \geq \exp(-4C_0V-6C_0(K_0+V)-10C_0C_3V).
\end{align*}

To conclude this section, we list and prove some desirable properties of the function $a$.
\begin{proposition}\label{aproperties}
There exists $\overline{C}, \overline{c} > 0$ such that, uniformly in $\delta$
\begin{align}\label{abounded}
\overline{c} \leq a(t,x) \leq  \overline{C} \text{ for any } (t,x) \in \R^+ \times \R.
\end{align}
Furthermore, for every time $t$ with a wave interaction, and almost every $x$:
\begin{align}\label{amonotone}
a(t+,x) \leq a(t-,x).
\end{align}
Moreover, for every time without a wave interaction, and for every $x$ such that a big shock is located at $x=x_i(t)$:
\begin{align}\label{abigjump}
\frac{a(t,x_i(t)+)}{a(t,x_i(t)-)}=
\begin{cases}
    C_1 & \text{ if $x_i$ is a $1$-shock,} \\
    \frac{1}{C_1} & \text{ if $x_i$ is a $2$-shock,} \\
\end{cases}
\end{align}
and if a small shock of strength $\sigma_i$ is located at $x=x_i(t)$:
\begin{align}\label{asmalljump}
\frac{a(t,x_i(t)+)}{a(t,x_i(t)-)}=\begin{cases}
    1-C_0\sigma_i & \text{ if $x_i$ is a $1$-shock,} \\
    1+C_0\sigma_i & \text{ if $x_i$ is a $2$-shock.} \\
\end{cases}
\end{align}
\end{proposition}
\begin{proof}
Firstly, we show \eqref{abounded}. For the lower bound, using Lemmas \ref{Lproperties} and \ref{Qproperties}, we find
\begin{align*}
a(t,x) \geq C_1^\Lambda k\prod_{i:\text{shock}}\xi_i(t,x).
\end{align*}
Now, when considering contributions to the weight due to shocks at time $t$, we need only consider weights corresponding to $1$-shocks for the lower bound. Using the fact that $||w_\delta(t,\cdot)||_{BV(\R)} \leq V$, we see that there can be maximally $\ceil{\frac{2V}{\eps}}$ big $1$-shocks in the solution at time $t$. Further, by Lemma \ref{productlemma}
\begin{align*}
\prod_{i:\text{small shock at time $t$}}\xi_i \geq \exp(-2V).
\end{align*}
Combining all of the contributions, we see
\begin{align*}
a(t,x) \geq C_1^{(\Lambda+\ceil{\frac{2V}{\eps}})}k\exp(-2V).
\end{align*}
For the upper bound, using the same logic and Lemma \ref{productlemma} gives
\begin{align*}
a(t,x) \leq \frac{\exp(V)}{C_1^{\ceil{\frac{2V}{\eps}}}}.
\end{align*}
This grants \eqref{abounded}. 
\par Next, we prove \eqref{amonotone}. The property \eqref{amonotone} obviously holds at times with trivial wave interactions, so it suffices to check each non-trivial wave interaction individually and show the property \eqref{amonotone} holds. For brevity, we only show it for interactions of types 1 to 5 and 9B (the remaining interactions of types 6 to 10 are easy to check because $L$ and $Q$ do not change at times with these interactions). Assume there is a non-trivial interaction at the point $(t^*,x^*)$. Then, we have the following.
\begin{enumerate}
    \item Assume the interaction is type 1. Then, $L(t+)=L(t-)+1$, so for $x < x^*$, $a(t+,x)=C_1a(t-,x)<a(t-,x)$ and for $x > x^*, a(t+,x)=a(t-,x)$.
    \item Assume the interaction is type 2. Then, $Q(t+)=(1-C_0(w_r-w_m))Q(t-)< Q(t-)$, so for $x < x^*$, $a(t+,x)< a(t-,x)$ and for $x > x^*, a(t+,x)=a(t-,x)$.
    \item Assume the interaction is of type 3A. Then, $Q(t+)=(\frac{1}{1+C_0(w_m'-w_r)})Q(t-)< Q(t-)$, so for $x < x^*$, $a(t+,x)<a(t-,x)$ and for $x > x^*, a(t+,x)=a(t-,x)$.
    \item Assume the interaction is of type 3B. Then, $L(t+)=L(t-)+1$, so for $x < x^*$, $a(t+,x)=C_1a(t-,x)<a(t-,x)$ and for $x > x^*$, $a(t+,x)=(1-C_0(w_m'-w_l))(1+C_0(w_m'-w_r))a(t-,x) <a(t-,x)$.
    \item Assume the interaction is of type 4A. Then
    \[Q(t+)=\min{(\frac{1-C_0(w_m-w_l)}{(1-C_0(w_m'-w_l))(1+C_0(w_m'-w_r))},1)}Q(t-) \leq Q(t-),\]
    so for $x<x^*$, $a(t+,x) \leq a(t-,x)$ and for $x > x^*$, $a(t+,x)=a(t-,x)$.
    \item Assume the interaction is of type 4B. Then, for $x < x^*$, $a(t+,x)=a(t-,x)$ and for $x > x^*$, 
    \begin{align*}a(t+,x)&=C_1a(t+,x^*-)=C_1a(t-,x^*-) \\
    &=\frac{C_1}{(1-C_0(w_m-w_l))(1-C_0(w_r-w_m))}a(t+,x) <a(t+,x).\end{align*}
    \item Assume the interaction is of type 5A. Then
    \[Q(t+)=\frac{1-C_0(w_m-w_l)}{(1-C_0(w_m'-w_l))(1+C_0(w_m'-w_r))}Q(t-)<Q(t-),\]
    so for $x<x^*, a(t+,x)<a(t-,x)$ and for $x > x^*$, $a(t+,x)=a(t-,x)$.
    \item Assume the interaction is of type 5B. Then, $Q(t+)=\frac{1-C_0(w_m-w_l)}{1+C_0(w_m'-w_r)}Q(t-)<Q(t-)$, so for $x<x^*, a(t+,x)<a(t-,x)$ and for $x > x^*$, $a(t+,x)=a(t-,x)$.
    \item Assume the interaction of of type 9B. Then, $L(t+)=L(t-)+1$, so for $x < x^*$, $a(t+,x)=C_1a(t-,x)<a(t-,x)$ and for $x > x^*, a(t+,x)=\frac{1}{C_1}a(t+,x^*-)=a(t-,x^*-) <a(t-,x)$.
\end{enumerate}
\par Finally, \eqref{abigjump} and \eqref{asmalljump} are satisfied by the definition of the weights $\xi_i$ as the coefficients $L$ and $Q$ remain constant in space for a fixed time.
\end{proof}

\section{Proof of Proposition \ref{mainprop}}\label{final}
In this section, we give the proof of Proposition \ref{mainprop}. We start with a lemma that will provide the framework to stitch together the relative entropy stability between wave interactions. The proof exactly follows that of Lemma 2.5 in \cite{MR3954680}, so we omit it.
\begin{lemma}[Lemma 2.5 in \cite{MR3954680}]\label{stitchinglemma}
Let $u \in L^\infty(\R^+ \times \R)$ satisfy \eqref{isothermal}\eqref{entropyinequality} with initial data $u^0$. Further, assume that $u$ verifies the Strong Trace property (Definition \ref{strongtrace}). Then, for all $v \in \statespace$, and all $c,d \in \R$ with $c < d$, the approximate left- and right-hand limits:
\begin{align*}
    \aplim_{t \to t_0^\pm}\int_c^d\eta(u(t,x)|v)\diff x,
\end{align*}
exist for all $t_0 \in (0,\infty)$ and verify:
\begin{align*}
\aplim_{t \to t_0^-}\int_c^d\eta(u(t,x)|v)\diff x \geq \aplim_{t \to t_0^+}\int_c^d\eta(u(t,x)|v)\diff x.
\end{align*}
Furthermore, the approximate right-hand limit exists at $t_0=0$ and verifies:
\begin{align*}
\int_c^d\eta(u^0(x)|v)\diff x \geq \aplim_{t \to t_0^+}\int_c^d\eta(u(t,x)|v)\diff x.
\end{align*}
\end{lemma}
The next lemma will control the boundaries of the cone of information.
\begin{lemma}[Lemma 7.2 in \cite{MR4487515}]\label{qcontrolbyeta}
There exists a constant $C > 0$ such that:
\[
|q(a;b)| \leq C\eta(a|b), \text{ for any } (a,b) \in \mathcal{D} \times \mathcal{D}.
\]
\end{lemma}
\begin{proof}
We have $q(b;b)=\partial_1q(b;b)=0$ for any $b \in \mathcal{D}$, so by Lemma \ref{quadraticentropy} and $q'' \in C^0(\mathcal{D})$, there exists a constant $C$ such that:
\[
|q(a;b)| \leq C|a-b|^2 \leq \frac{C}{c^*}\eta(a|b).
\]
\end{proof}
Finally, we give the proof of Proposition \ref{mainprop}. Fix an initial data $u^0$ verifying the hypotheses, which fixes the domain $\mathcal{D}$. Let $0<\eps<\frac{1}{2}$ be sufficiently small so that Proposition \ref{shockstability} is verified. Then, fix the value $l$ to be greater than the maximal speed $\hat{\lambda}$ from Proposition \ref{shockstability} and the constant $C$ from Lemma \ref{qcontrolbyeta}. Now, consider the family of solutions $\psi_\delta$ constructed by the modified front tracking scheme. Fix $T,R>0$, and $p \in \mathbb{N}$. Take $\delta$ sufficiently small so that
\begin{align}\label{initdataapprox}
    ||u^0-\psi_\delta(0,\cdot)||_{L^2([-R,R])} \leq \frac{1}{p}.
\end{align}
For convenience, denote by $\psi$ the associated front tracking solution $\psi_\delta$. In particular, it verifies Lemma \ref{wellposed}. This immediately grants the first two properties in Proposition \ref{mainprop}. So, we need only show the $L^2$-stability. 
\par Consider two successive interaction times $t_j < t_{j+1}$ in the modified front tracking scheme that produced $\psi$. Let the curves of discontinuity between two times $t_j, t_{j+1}$ be defined as $h_1, ..., h_N$, where $N \in \N$, so that:
\begin{align*}
h_1(t) < ... < h_N(t),
\end{align*}
for all $t \in (t_j, t_{j+1})$. Define the boundaries of the cone of information:
\begin{align*}
    &h_0(t)=-R+lt, \\
    &h_{N+1}(t)=R-lt.
\end{align*}
Crucially there are no interactions between waves in $\psi$ and the cone of information. Further, for any $t \in [t_j, t_{j+1}]$, note that on $Q:=\{(r,x)|t_j < r < t, h_i(r) < x < h_{i+1}(r)\}$, the functions $\psi(r,x)$ and $a(r,x)$ are constant. Integrating \eqref{relativeentropic} on $Q$, and using Definition \ref{strongtrace}, we obtain for any $u \in \weak$
\begin{align*}
&\aplim_{s \to t_-}\int_{h_i(t)}^{h_{i+1}(t)}a(t-,x)\eta(u(s,x)|\psi(t,x))\diff x \\
&\leq \aplim_{s \to t_j^+}\int_{h_i(t_j)}^{h_{i+1}(t_j)}a(t_j+,x)\eta(u(s,x)|\psi(t_j,x))\diff x +\int_{t_j}^t(F_i^+(r)-F_{i+1}^-(r))\diff r,
\end{align*}
where
\begin{align*}
&F_i^+(r)=a(r,h_i(r)+)\left(q(u(r,h_i(r)+);\psi(r,h_i(r)+))-\dot{h}_i(r)\eta(u(r,h_i(r)+)|\psi(r,h_i(r)+))\right), \\
&F_i^-(r)=a(r,h_i(r)-)\left(q(u(r,h_i(r)-);\psi(r,h_i(r)-))-\dot{h}_i(r)\eta(u(r,h_i(r)-)|\psi(r,h_i(r)-))\right).
\end{align*}
Summing in $i$, and combining terms corresponding to $i$ in one sum and terms corresponding to $i+1$ in another, we obtain
\begin{align*}
&\aplim_{s \to t_-}\int_{-R+vt}^{R-vt}a(t-,x)\eta(u(s,x)|\psi(t,x))\diff x \\
&\leq \aplim_{s \to t_j^+}\int_{R+vt_j}^{R-vt_j}a(t_j+,x)\eta(u(s,x)|\psi(t_j,x))\diff x +\sum_{i=1}^N\int_{t_j}^t(F_i^+(r)-F_{i+1}^-(r))\diff r,
\end{align*}
where we have used that $F_0^+ \leq 0$ and $F_{N+1}^- \geq 0$ due to Lemma \ref{qcontrolbyeta}, the definition of $l$, and $\dot{h}_0=-l=-\dot{h}_{N+1}$.
\par We decompose the second term on the right-hand side into two sums, one corresponding to shocks and one corresponding to rarefactions. Due to Proposition \ref{shockstability} and \eqref{asmalljump}, \eqref{abigjump}, for any $i$ corresponding to a shock
\begin{align*}
F_i^+(r)-F_i^-(r) \leq 0, \text{ for almost every } t_j < r < t.
\end{align*}
Denote $\mathcal{R}$ to be the set of $i$ corresponding to rarefactions. Then, for any $i \in \mathcal{R}$, $a(r,h_i(r)+)=a(r,h_i(r)-)$. From Proposition \ref{rarefactionstability}, and by taking $\delta$ sufficiently small so that $\kappa \leq \frac{1}{\Theta \overline{C} VpT}$ (this is possible due to the uniform continuity of $(w,v) \mapsto \exp(w)$ and $(\tau,v) \mapsto \frac{1}{\tau}$ on compact subsets of $\R^+ \times \R$, and the structure of the rarefaction curves):
\[
\sum_{i \in \mathcal{R}}\int_{t_j}^t(F_i^+(r)-F_i^-(r))\diff r \leq \Theta \overline{C} \kappa (t-t_j)\sum_{i \in \mathcal{R}}\sigma_i \leq \frac{(t-t_j)}{pT}.
\]
In total, we find:
\begin{align*}
&\aplim_{s \to t_-}\int_{-R+vt}^{R-vt}a(t-,x)\eta(u(s,x)|\psi(t,x))\diff x \\
&\leq \aplim_{s \to t_j^+}\int_{R+vt_j}^{R-vt_j}a(t_j+,x)\eta(u(s,x)|\psi(t_j,x))\diff x +\frac{(t-t_j)}{pT}.
\end{align*}
Consider any $0 < t < T$, and denote $0=t_0 < t_1, ..., t_J<t_{J+1}=t$ the times of wave interactions in the interval $[0,t]$. Lemma \ref{stitchinglemma} and \eqref{amonotone} grant:
\begin{align*}
&\int_{-R+lt}^{R-lt}a(t,x)\eta(u(t,x)|\psi(t,x))\diff x -\int_{-R}^Ra(0,x)\eta(u(0,x)|\psi(0,x))\diff x \\
&\leq \aplim_{s \to t^+}\int_{-R+lt}^{R-lt}a(t,x)\eta(u(t,x)|\psi(t,x))\diff x -\int_{-R}^Ra(0,x)\eta(u(0,x)|\psi(0,x))\diff x \\
& \leq \sum_{j=1}^{J+1}\Bigg(\aplim_{s \to t_j^+}\int_{-R+lt_j}^{R-lt_j}a(t_j-,x)\eta(u(t,x)|\psi(t,x))\diff x \\
&\phantom{{}\leq \sum_{j=1}^{J+1}(}-\aplim_{s \to t_{j-1}^+}\int_{-R+lt_{j-1}}^{R-lt_{j-1}}a(t_{j-1}-,x)\eta(u(0,x)|\psi(0,x))\diff x\Bigg) \\
& \leq \sum_{j=1}^{J+1}\Bigg(\aplim_{s \to t_j^-}\int_{-R+lt_j}^{R-lt_j}a(t_j-,x)\eta(u(t,x)|\psi(t,x))\diff x \\
&\phantom{{}\leq \sum_{j=1}^{J+1}(}-\aplim_{s \to t_{j-1}^+}\int_{-R+lt_{j-1}}^{R-lt_{j-1}}a(t_{j-1}-,x)\eta(u(0,x)|\psi(0,x))\diff x\Bigg) \\
& \leq \sum_{j=1}^{J+1}\Bigg(\aplim_{s \to t_j^-}\int_{-R+lt_j}^{R-lt_j}a(t_j-,x)\eta(u(t,x)|\psi(t,x))\diff x \\
&\phantom{{}\leq \sum_{j=1}^{J+1}(}-\aplim_{s \to t_{j-1}^+}\int_{-R+lt_{j-1}}^{R-lt_{j-1}}a(t_{j-1}+,x)\eta(u(0,x)|\psi(0,x))\diff x\Bigg) \\
&\leq \sum_{j=1}^{J+1}\frac{(t_j-t_{j-1})}{pT} \leq \frac{1}{p}.
\end{align*}
Using  Lemma \ref{quadraticentropy}, \eqref{abounded}, and \eqref{initdataapprox}, we find:
\[
\overline{c}c^*||\psi(t,\cdot)-u(t,\cdot)||_{L^2([-R+vt, R-vt])}^2 \leq 2\overline{C}c^{**}||u^0-u(0,\cdot)||_{L^2([-R,R])}^2+\frac{1+2\overline{C}c^{**}}{p},
\]
where $\overline{c},\overline{C}$ are the constants from Proposition \ref{aproperties} and $c^*,c^{**}$ are the constants from Lemma \ref{quadraticentropy}. Taking $p$ sufficiently large so that $\frac{1+2\overline{C}c^{**}}{\overline{c}c^*p} < \frac{1}{m}$ grants the result. 

\appendix

\section{Proof of Theorem \ref{maineulerian}.}\label{appendixeulerian}
For the sake of completeness, we give here a proof of Theorem \ref{maineulerian}, assuming that one is well-acquainted with the proof of Proposition \ref{mainprop}. Let $U_n=(\rho_n, \rho_nv_n)$ be a sequence in $\weak^E$ such that $U_n^0 \to U$ in $L^2(\R)$, where $U=(\rho, \rho v) \in \weak^E \cap L^\infty(\R^+;BV(\R))$. By Theorem 1 in \cite{MR885816}, for each $n$ there exists a bilipschitz homeomorphism $X_n:\R^+ \times \R \to \R^+ \times \R$ defined by $X_n(t,y)=(t, \phi_n(t,y))$ such that
\begin{align*}
    u_n(t,x):=(\frac{1}{\rho_n(t,\phi_n^{-1}(t,x))}, v_n(t,\phi_n^{-1}(t,x)))
\end{align*}
is a weak solution to \eqref{isothermal} in $\weak$. Owing to the fact that $U_n \in \weak^E$, we have uniform bounds
\begin{align}
\tilde{c} \leq &\frac{\partial \phi_n}{\partial y} \leq \tilde{C}, \label{deriv1} \\
|&\frac{\partial \phi_n}{\partial t}| \leq \tilde{C}\tilde{M}. \label{deriv2}
\end{align}
Further, there exists a change of variables $X(t,y)=(t,\phi(t,y))$ such that the same holds for $U$ and its associated Lagrangian solution 
\begin{align*}
    u(t,x):=(\frac{1}{\rho(t,\phi^{-1}(t,x))}, v(t,\phi^{-1}(t,x)))).
\end{align*}
One may also check that $u \in \weak \cap L^\infty(\R^+;BV(\R))$, and $u_n^0 \to u^0$ in $L^2(\R)$. So, fix $T,R >0$. By examining the proof of Proposition \ref{mainprop}, we see that for any $R'>0$, we may obtain for any $n>0$, $0 < t < T$, and $m > 0$, there exists a function $\psi_n \in L^\infty([0,T];L^\infty(\R) \cap BV(\R))$ such that
\begin{align*}
\int_{-R'+lt}^{R'-lt}\eta(u_n(t,x)|\psi_n(t,x))\diff x \leq \frac{1}{m}.
\end{align*}
Further, there is a uniform bound on the $L^\infty$ norm and total variation in $n$, and the speed of waves in the functions $\psi_n$ is uniformly bounded by $\hat{\lambda}$. Now, using Proposition 6.1 in \cite{MR3322780}, this translates exactly to
\begin{align*}
\int_{-R}^{R}\tilde{\eta}(U_n(t,y)|\zeta_n(t,y)) \diff y \leq \frac{1}{m},
\end{align*}
where $\zeta_n(t,y)=\psi_n(t,\phi_n(t,y))$, by taking $R'$ sufficiently large. Using Lemma \ref{quadraticentropy}, this implies
\begin{align*}
||U_n(t,\cdot)-\zeta_n(t,\cdot)||_{L^2([-R,R])} \leq \frac{1}{n},
\end{align*}
by taking $m$ sufficiently large. It remains to show that the sequence $\{\zeta_n\}_{n=1}^\infty$ verifies the hypotheses of Theorem \ref{compactness}. Firstly, the uniform BV bound follows from the fact that the functions $\psi_n$ are uniformly bounded in $BV$. Finally, the proof that the functions $\zeta_n$ are Lipschitz in time in $L^1$ follows from the following lemma, which says that the curves of discontinuity in $\zeta_n$, which are the images of curves of discontinuity in $\psi_n$ under $X_n^{-1}$, retain the finite speed of propagation property.

\begin{lemma}\label{finitespeed}
There exists a constant $\hat{\Lambda}$ such that the following holds. Let $h:[a,b] \to \R$ be a segment of a curve of discontinuity in $\zeta_n$. Let $0 < a < s < t < b <T$, and let $(s, y_1), (t, y_2)$ be two points on $h$. Then, we have
\begin{align*}
\frac{|y_2-y_1|}{|t-s|} \leq \hat{\Lambda}.
\end{align*}
\end{lemma}
\begin{proof}
Let $(s,x_1)$ and $(t,x_2)$ be the images of $(s,y_1)$ and $(t,y_2)$ under $X_n$ respectively. Note that these lie on $X_n(h(t))$, which is a curve of discontinuity for the function $\psi_n$. Firstly, consider the auxiliary point $(s,x_2)$. Let $(s, y_3)$ be the preimage of $(s,x_2)$ under $X_n$. Then, due to \eqref{deriv1}, we see
\begin{align}\label{apppart1}
|y_1-y_3| \leq \frac{2}{\tilde{c}}|x_1-x_2|. 
\end{align}
Now, consider the point $(t,y_3)$, and let $(t, x_4)$ be its image under $X_n$. Due to \eqref{deriv2}, we see
\begin{align}
|x_4-x_2| \leq \tilde{C}\tilde{M}|t-s|.
\end{align}
Using $\eqref{deriv1}$ again gives
\begin{align}\label{apppart2}
|y_3-y_2| \leq \frac{2}{\tilde{c}}|x_4-x_2| \leq 2\frac{\tilde{C}\tilde{M}}{\tilde{c}}|t-s|.
\end{align}
Together, \eqref{apppart1} and \eqref{apppart2} give
\begin{align*}
|y_1-y_2| \leq  \frac{2}{\tilde{c}}|x_1-x_2|+2\frac{\tilde{C}\tilde{M}}{\tilde{c}}|t-s|.
\end{align*}
Dividing both sides by $|t-s|$ gives the result, as $\frac{|x_1-x_2|}{|t-s|} \leq \hat{\lambda}$, the maximal speed of propagation in the $(t,x)$-plane.
\end{proof}
Using Lemma \ref{finitespeed}, the uniform Lipschitz in time bound follows exactly as in Section 7.4 in \cite{MR1816648}, or the proof of Proposition 2.2 in \cite{MR3071091}. From this point, the proof may be completed exactly as the proof of Theorem \ref{main}, as the sequence $\{\zeta_n\}_{n=1}^\infty$ verifies the hypotheses of Theorem \ref{compactness}.

\section{Proof of Lemma \ref{wellposed}: Well-posedness of the scheme}\label{appendix}
In this section, we prove Lemma \ref{wellposed}. The first part of the proof is similar to the proof of $L^\infty$ and $BV$ estimates for approximate solutions to \eqref{isothermal} generated by the Glimm scheme stated in Theorem 7.2 in \cite{MR1942468}. Let $w^0_\delta=-ln(\tau^0_\delta)$. As the initial density is bounded away from vacuum and infinity, and is BV, we have $-D \leq w^0_\delta \leq D$ for some $D>0$, and $||w^0_\delta||_{BV(\R)} < \infty$. Then, due to the $L^\infty$ bound on $w^0_\delta$ and the structure of the wave curves, solving the Riemann problems at $t=0$ grants:
\begin{align*}
    ||w_\delta(0+,\cdot)||_{BV(\R)} \leq C_1||w^0_\delta||_{BV(\R)}.
\end{align*}
Then, using Lemma \ref{tvd}:
\begin{align*}
    ||w_\delta(t,\cdot)||_{BV(\R)} \leq ||w_\delta(0+,\cdot)||_{BV(\R)} \leq C_1||w^0_\delta||_{BV(\R)}.
\end{align*}
Thus, there exists $K_1 > 0$ such that $|w_\delta(t,x)| \leq K_1$. Further, one has
\begin{align*}
    ||\tau_\delta(t,\cdot)||_{BV(\R)} \leq C_2||w_\delta(t,\cdot)||_{BV(\R)} \leq C_2C_1||w^0_\delta||_{BV(\R)},
\end{align*}
and
\begin{align*}
    ||v_\delta(t,\cdot)||_{BV(\R)} \leq C_3||w_\delta(t,\cdot)||_{BV(\R)} \leq C_3C_1||w^0_\delta||_{BV(\R)},
\end{align*}
from the equations of the elementary wave curves \eqref{wavecurves} and $|W'(s)| \leq C_3$ for $|s| \leq 2K_1$. Therefore:
\begin{align*}
    &|v_\delta(t,x)| \leq K_2, \indent (t,x) \in [0,T] \times \R, \\
    &||\psi_\delta(t,\cdot)||_{BV(\R)} \leq C_4||w^0_\delta||_{BV(\R)} \leq C_5||\psi^0_\delta||_{BV(\R)} \leq C_6||u^0||_{BV(\R)}.
\end{align*}
We obtain that $\psi_\delta(t,\cdot) \in \mathcal{D}$ for any $x \in \R$, where we may define
\begin{align*}
\mathcal{D}:=\{(\tau, v) \in \R^+ \times \R): \exp(-K_1) \leq \tau \leq \exp(K_1), |v| \leq K_2\}.
\end{align*}
Further, the Lipschitz bound in time in $L^1$ follows from the fact that the wave speeds are bounded by $\hat{\lambda}$ (a proof of this may be found in Section 7.4 in \cite{MR1816648}, or Proposition 2.2 in \cite{MR3071091}).
\par Finally, we need to show that the number of waves in $[0,T] \times \R$ is finite. It suffices to show that, for a fixed $\delta > 0$, there is a finite number of interactions where more than one outgoing wave of each family is generated (however, the number may blow up as $\delta \to 0$). Clearly, the number of interactions at $t=0$ is finite. After that, we need only show that the number of possible interactions of types 1, 2, 4A, 4B, 6, 7, 9A, and 9B that create more than one rarefaction at times $t=t^* > 0$ is finite. Without loss of generality, we will assume it is of type 1, 2, 4A, or 4B. 
\par Firstly, assume that the interaction is of type 1. Then, whenever a type 1 interaction occurs, the number of big shocks goes down by 1. However, the number of times a big shock can be created is finite. Indeed, if a big shock is created, in either a type 4B or 9B interaction, $\Delta S(t^*)\leq -\eps$. By Lemma \ref{Sdissipationbd}, we conclude that this can only happen at most $\ceil{\frac{K_0+V}{\eps}}$ times. Thus, the number of type 1 interactions is bounded by $\ceil{\frac{V}{\eps}}+\ceil{\frac{K_0+V}{\eps}}$.
\par Next, assume the interaction is of type 2. Then, if more than one rarefaction is created, we must have $w_r-w_m' > \delta$. However, this implies $\Delta S(t^*)=-(w_r-w_m) \leq -\delta$. So, by Lemma \ref{Sdissipationbd}, this may not happen more than $\ceil{\frac{K_0+V}{\delta}}$ times.
\par Next, assume the interaction is of type 4A. Then, similar to the type 2 case, we see $\Delta S(t^*) \leq -\delta$. this may not happen more than $\ceil{\frac{K_0+V}{\delta}}$ times.
\par Finally, assume the interaction is of type 4B. Then, by Lemma \ref{Sdissipationbd}, this may not happen more than $\ceil{\frac{K_0+V}{\eps}}$ times.
\par To conclude, we use the following lemma.
\begin{lemma}[Lemma 2.5 in \cite{MR1820292}]
Let a front tracking pattern be given in $[0,T) \times \R$, made of segments of the two families. Assume that the velocities of the first family lie between constants $a_1 < a_2$, and the velocities of the second family lie between constants $b_1 < b_2$, with $a_2 < b_1$. Assume that the pattern has the following properties:
\begin{enumerate}
    \item at $t=0$ there is a finite number of waves,
    \item interactions occur only between two wave fronts,
    \item except for a finite number of interactions, there is at most one outgoing wave of each family for each interaction.
\end{enumerate}
Then, the number of waves in $[0,T) \times \R$ is finite. 
\end{lemma}
So, applying the lemma, we obtain the well-posedness.
\bibliographystyle{plain}
\bibliography{refs}

\begin{thebibliography}{10}

\bibitem{MR1820292}
Debora Amadori and Graziano Guerra.
\newblock Global {BV} solutions and relaxation limit for a system of conservation laws.
\newblock {\em Proc. Roy. Soc. Edinburgh Sect. A}, 131(1):1--26, 2001.

\bibitem{MR2126567}
Fumioki Asakura.
\newblock Wave-front tracking for the equations of isentropic gas dynamics.
\newblock {\em Quart. Appl. Math.}, 63(1):20--33, 2005.

\bibitem{MR279443}
Nikolai~S. Bahkvalov.
\newblock The existence in the large of a regular solution of a quasilinear hyperbolic system.
\newblock {\em \v Z. Vy\v cisl. Mat i Mat. Fiz.}, 10:969--980, 1970.

\bibitem{MR3001706}
Paolo Baiti and Edda Dal~Santo.
\newblock Front tracking for a {$2\times 2$} system of conservation laws.
\newblock {\em Electron. J. Differential Equations}, pages No. 220, 14, 2012.

\bibitem{MR2150387}
Stefano Bianchini and Alberto Bressan.
\newblock Vanishing viscosity solutions of nonlinear hyperbolic systems.
\newblock {\em Ann. of Math. (2)}, 161(1):223--342, 2005.

\bibitem{MR1816648}
Alberto Bressan.
\newblock {\em Hyperbolic systems of conservation laws}, volume~20 of {\em Oxford Lecture Series in Mathematics and its Applications}.
\newblock Oxford University Press, Oxford, 2000.
\newblock The one-dimensional Cauchy problem.

\bibitem{MR1686652}
Alberto Bressan, Graziano Crasta, and Benedetto Piccoli.
\newblock Well-posedness of the {C}auchy problem for {$n\times n$} systems of conservation laws.
\newblock {\em Mem. Amer. Math. Soc.}, 146(694):viii+134, 2000.

\bibitem{MR4661213}
Alberto Bressan and Camillo De~Lellis.
\newblock A remark on the uniqueness of solutions to hyperbolic conservation laws.
\newblock {\em Arch. Ration. Mech. Anal.}, 247(6):Paper No. 106, 12, 2023.

\bibitem{MR1701818}
Alberto Bressan and Paola Goatin.
\newblock Oleinik type estimates and uniqueness for {$n\times n$} conservation laws.
\newblock {\em J. Differential Equations}, 156(1):26--49, 1999.

\bibitem{MR4690554}
Alberto Bressan and Graziano Guerra.
\newblock Unique solutions to hyperbolic conservation laws with a strictly convex entropy.
\newblock {\em J. Differential Equations}, 387:432--447, 2024.

\bibitem{MR1489317}
Alberto Bressan and Philippe LeFloch.
\newblock Uniqueness of weak solutions to systems of conservation laws.
\newblock {\em Arch. Rational Mech. Anal.}, 140(4):301--317, 1997.

\bibitem{MR1757395}
Alberto Bressan and Marta Lewicka.
\newblock A uniqueness condition for hyperbolic systems of conservation laws.
\newblock {\em Discrete Contin. Dynam. Systems}, 6(3):673--682, 2000.

\bibitem{MR1723032}
Alberto Bressan, Tai-Ping Liu, and Tong Yang.
\newblock {$L^1$} stability estimates for {$n\times n$} conservation laws.
\newblock {\em Arch. Ration. Mech. Anal.}, 149(1):1--22, 1999.

\bibitem{MR3040678}
Geng Chen and Helge~Kristian Jenssen.
\newblock No {TVD} fields for 1-{D} isentropic gas flow.
\newblock {\em Comm. Partial Differential Equations}, 38(4):629--657, 2013.

\bibitem{MR4487515}
Geng Chen, Sam~G. Krupa, and Alexis~F. Vasseur.
\newblock Uniqueness and weak-{BV} stability for {$2\times 2$} conservation laws.
\newblock {\em Arch. Ration. Mech. Anal.}, 246(1):299--332, 2022.

\bibitem{MR924671}
Gui~Qiang Chen.
\newblock Convergence of the {L}ax-{F}riedrichs scheme for isentropic gas dynamics. {III}.
\newblock {\em Acta Math. Sci. (English Ed.)}, 6(1):75--120, 1986.

\bibitem{MR1942468}
Gui-Qiang Chen and Dehua Wang.
\newblock The {C}auchy problem for the {E}uler equations for compressible fluids.
\newblock In {\em Handbook of mathematical fluid dynamics, {V}ol. {I}}, pages 421--543. North-Holland, Amsterdam, 2002.

\bibitem{concaveconvex}
Jeffrey Cheng.
\newblock L$^2$-stability \& minimal entropy conditions for scalar conservation laws with concave-convex fluxes.
\newblock {\em Preprint on Arxiv}, 2024.

\bibitem{MR1641709}
Rinaldo~M. Colombo and Nils~H. Risebro.
\newblock Continuous dependence in the large for some equations of gas dynamics.
\newblock {\em Comm. Partial Differential Equations}, 23(9-10):1693--1718, 1998.

\bibitem{MR303068}
Constantine~M. Dafermos.
\newblock Polygonal approximations of solutions of the initial value problem for a conservation law.
\newblock {\em J. Math. Anal. Appl.}, 38:33--41, 1972.

\bibitem{MR546634}
Constantine~M. Dafermos.
\newblock The second law of thermodynamics and stability.
\newblock {\em Arch. Rational Mech. Anal.}, 70(2):167--179, 1979.

\bibitem{MR3468916}
Constantine~M. Dafermos.
\newblock {\em Hyperbolic conservation laws in continuum physics}, volume 325 of {\em Grundlehren der mathematischen Wissenschaften [Fundamental Principles of Mathematical Sciences]}.
\newblock Springer-Verlag, Berlin, fourth edition, 2016.

\bibitem{MR922139}
Xia~Xi Ding, Gui~Qiang Chen, and Pei~Zhu Luo.
\newblock Convergence of the {L}ax-{F}riedrichs scheme for isentropic gas dynamics. {I}, {II}.
\newblock {\em Acta Math. Sci. (English Ed.)}, 5(4):415--432, 433--472, 1985.

\bibitem{MR523630}
Ronald~J. DiPerna.
\newblock Uniqueness of solutions to hyperbolic conservation laws.
\newblock {\em Indiana Univ. Math. J.}, 28(1):137--188, 1979.

\bibitem{MR719807}
Ronald~J. DiPerna.
\newblock Convergence of the viscosity method for isentropic gas dynamics.
\newblock {\em Comm. Math. Phys.}, 91(1):1--30, 1983.

\bibitem{MR194770}
James Glimm.
\newblock Solutions in the large for nonlinear hyperbolic systems of equations.
\newblock {\em Comm. Pure Appl. Math.}, 18:697--715, 1965.

\bibitem{MR4667839}
William~M. Golding, Sam~G. Krupa, and Alexis~F. Vasseur.
\newblock Sharp {$a$}-contraction estimates for small extremal shocks.
\newblock {\em J. Hyperbolic Differ. Equ.}, 20(3):541--602, 2023.

\bibitem{MR1970885}
Feimin Huang and Zhen Wang.
\newblock Convergence of viscosity solutions for isothermal gas dynamics.
\newblock {\em SIAM J. Math. Anal.}, 34(3):595--610, 2002.

\bibitem{MR3519973}
Moon-Jin Kang and Alexis~F. Vasseur.
\newblock Criteria on contractions for entropic discontinuities of systems of conservation laws.
\newblock {\em Arch. Ration. Mech. Anal.}, 222(1):343--391, 2016.

\bibitem{MR4184662}
Sam~G. Krupa.
\newblock Finite time stability for the {R}iemann problem with extremal shocks for a large class of hyperbolic systems.
\newblock {\em J. Differential Equations}, 273:122--171, 2021.

\bibitem{MR3954680}
Sam~G. Krupa and Alexis~F. Vasseur.
\newblock On uniqueness of solutions to conservation laws verifying a single entropy condition.
\newblock {\em J. Hyperbolic Differ. Equ.}, 16(1):157--191, 2019.

\bibitem{MR2771666}
Nicholas Leger.
\newblock {$L^2$} stability estimates for shock solutions of scalar conservation laws using the relative entropy method.
\newblock {\em Arch. Ration. Mech. Anal.}, 199(3):761--778, 2011.

\bibitem{MR2807139}
Nicholas Leger and Alexis~F. Vasseur.
\newblock Relative entropy and the stability of shocks and contact discontinuities for systems of conservation laws with non-{BV} perturbations.
\newblock {\em Arch. Ration. Mech. Anal.}, 201(1):271--302, 2011.

\bibitem{MR1383202}
Pierre-Louis Lions, Beno\^it Perthame, and Panagiotis~E. Souganidis.
\newblock Existence and stability of entropy solutions for the hyperbolic systems of isentropic gas dynamics in {E}ulerian and {L}agrangian coordinates.
\newblock {\em Comm. Pure Appl. Math.}, 49(6):599--638, 1996.

\bibitem{MR1284790}
Pierre-Louis Lions, Benoit Perthame, and Eitan Tadmor.
\newblock Kinetic formulation of the isentropic gas dynamics and {$p$}-systems.
\newblock {\em Comm. Math. Phys.}, 163(2):415--431, 1994.

\bibitem{MR236526}
Takaaki Nishida.
\newblock Global solution for an initial boundary value problem of a quasilinear hyperbolic system.
\newblock {\em Proc. Japan Acad.}, 44:642--646, 1968.

\bibitem{MR1359913}
Fr\'ed\'eric Poupaud, Michel Rascle, and Jean-Paul Vila.
\newblock Global solutions to the isothermal {E}uler-{P}oisson system with arbitrarily large data.
\newblock {\em J. Differential Equations}, 123(1):93--121, 1995.

\bibitem{MR1630338}
Bernhard Riemann.
\newblock The propagation of planar air waves of finite amplitude [{A}bh. {G}es. {W}iss. {G}\"ottingen {\bf 8} (1860), 43--65].
\newblock In {\em Classic papers in shock compression science}, High-press. Shock Compression Condens. Matter, pages 109--128. Springer, New York, 1998.

\bibitem{MR3071091}
N.~H. Risebro and F.~Weber.
\newblock A note on front tracking for the {K}eyfitz-{K}ranzer system.
\newblock {\em J. Math. Anal. Appl.}, 407(2):190--199, 2013.

\bibitem{MR3322780}
Denis Serre and Alexis~F. Vasseur.
\newblock {$L^2$}-type contraction for systems of conservation laws.
\newblock {\em J. \'Ec. polytech. Math.}, 1:1--28, 2014.

\bibitem{MR584398}
Luc Tartar.
\newblock Compensated compactness and applications to partial differential equations.
\newblock In {\em Nonlinear analysis and mechanics: {H}eriot-{W}att {S}ymposium, {V}ol. {IV}}, volume~39 of {\em Res. Notes in Math.}, pages 136--212. Pitman, Boston, Mass.-London, 1979.

\bibitem{MR2508169}
Alexis~F. Vasseur.
\newblock Recent results on hydrodynamic limits.
\newblock In {\em Handbook of differential equations: evolutionary equations. {V}ol. {IV}}, Handb. Differ. Equ., pages 323--376. Elsevier/North-Holland, Amsterdam, 2008.

\bibitem{MR885816}
David~H. Wagner.
\newblock Equivalence of the {E}uler and {L}agrangian equations of gas dynamics for weak solutions.
\newblock {\em J. Differential Equations}, 68(1):118--136, 1987.

\end{thebibliography}
\end{document}